\documentclass[10pt, a4paper,reqno]{amsart}
\usepackage{amsthm,euscript,colonequals}
\usepackage{stmaryrd}

\newtheorem{thm}{Theorem}[section]

\newtheorem{lem}[thm]{Lemma}
\newtheorem{cor}[thm]{Corollary}

\theoremstyle{definition} 
\newtheorem{dfn}[thm]{Definition}

\newtheorem{example}[thm]{Example}

\newtheorem{exa}[thm]{Example}
\newtheorem{rem}[thm]{Remark}

\newtheorem{nota}[thm]{Notation}

\usepackage[top=30truemm,bottom=30truemm,left=35truemm,right=35truemm]{geometry}
\usepackage{amsmath,amssymb,amscd,bm,graphicx,tikz-cd,upgreek,dsfont,amsfonts,tikz}
\usetikzlibrary{cd,calc}
\usepackage[pagebackref]{hyperref}

\hypersetup{
  colorlinks   = true,          
  urlcolor     = blue,          
  linkcolor    = purple,          
  citecolor   = blue             
}

\definecolor{Pink}{RGB}{230 56 243}
\definecolor{Blue}{RGB}{0 19 147}
\definecolor{Green}{RGB}{66 147 41}
\definecolor{Grey}{RGB}{102 102 102}
\definecolor{Orange}{RGB}{237 107 45}
\definecolor{Red}{RGB}{234 53 159}

\def\mythick{0.5}

\tikzset{
         DWs/.style={circle,draw=black,circle,fill=white,inner sep=0pt, minimum size=3pt},
W/.style={circle,draw=black,circle,fill=white,inner sep=0pt, minimum size=4pt},
B/.style={circle,draw=black!80!white,circle,fill=black!80!white,inner sep=0pt, outer sep=4pt, minimum size=3pt},
Bs/.style={circle,draw=black!80!white,circle,fill=black!80!white,inner sep=0pt, outer sep=2pt, minimum size=3pt},
BL/.style={circle,draw=blue!60!white,circle,fill=blue!60!white,inner sep=0pt, minimum size=4pt},
R/.style={circle,draw=red!60!white,circle,fill=red!60!white,inner sep=0pt, minimum size=4pt},  
G/.style={circle,draw=green!65!black,circle,fill=green!65!black,inner sep=0pt, minimum size=4pt},     
Rs/.style={circle,draw=red!60!white,circle,fill=red!60!white,inner sep=0pt, minimum size=2pt}, 
BLs/.style={circle,draw=blue!60!white,circle,fill=blue!60!white,inner sep=0pt, minimum size=2pt},
Gs/.style={circle,draw=green!65!black,circle,fill=green!65!black,inner sep=0pt, minimum size=2pt},  }

\usepackage{setspace}
\newcommand{\marginparstretch}{0.6}
\let\oldmarginpar\marginpar
\renewcommand\marginpar[1]{\-\oldmarginpar[\framebox{\setstretch{\marginparstretch}\begin{minipage}{\marginparwidth}{\raggedleft\tiny #1}\end{minipage}}]{\framebox{\setstretch{\marginparstretch}\begin{minipage}{\marginparwidth}{\raggedright\tiny #1}\end{minipage}}}}

\usepackage{enumitem}
\setlist[enumerate]{format=\normalfont}

\usepackage{array,multirow,booktabs,longtable}
\numberwithin{equation}{section}

\newcommand{\ii}{\kern 1pt {\rm i}\kern 1pt }

\newcommand{\Spec}{\operatorname{Spec}}

\newcommand{\Hom}{\operatorname{Hom}}

\newcommand{\Ext}{\operatorname{Ext}}

\newcommand{\RHom}{\operatorname{{\mathbf R}Hom}}
\newcommand{\End}{\operatorname{End}}

\newcommand{\Auteq}{\operatorname{Auteq}}

\newcommand{\coh}{\operatorname{coh}}

\newcommand{\fmod}{\operatorname{mod}}

\newcommand{\proj}{\operatorname{proj}}

\newcommand{\Kb}{\operatorname{K^b}}

\newcommand{\CM}{\operatorname{CM}}

\newcommand\Db{\mathop{\rm{D}^b}}

\newcommand{\LatticePoint}[1][]{%
\begin{tikzpicture}
\filldraw (0,0) circle (2pt);
\end{tikzpicture}
}

\newcommand{\Stab}{\operatorname{Stab}}
\newcommand{\cStab}[1]{\mathrm{Stab}_{#1}^{\kern -0.5pt \circ}\kern -0.2pt}
\newcommand{\nStab}[1]{\mathrm{Stab}_{#1}\kern -0.1pt}
\newcommand{\TitsR}{{\sf Tits}_{\kern 1pt \bR}}
\newcommand{\CTitsR}{{\sf CTits}_{\kern 1pt \bR}}
\newcommand{\CTitsC}{{\sf CTits}_{\kern 1pt \bC}}

\newcommand{\cAut}[1]{\mathrm{Aut}_{#1}^{\kern -0.5pt \circ}\kern -0.2pt}

\def\Br{\mathop{\sf Br}\nolimits}
\def\PBr{\mathop{\sf PBr}\nolimits}

\newcommand{\bC}{\mathbb{C}}

\newcommand{\bR}{\mathbb{R}}

\newcommand{\Flop}{\operatorname{\sf Flop}}

\newcommand{\scrA}{\EuScript{A}}

\newcommand{\scrC}{\EuScript{C}}

\newcommand{\scrG}{\EuScript{G}}
\newcommand{\scrH}{\EuScript{H}}
\newcommand{\scrI}{\EuScript{I}}
\newcommand{\scrJ}{\EuScript{J}}
\newcommand{\scrK}{\EuScript{K}}

\newcommand{\scrN}{\EuScript{N}}
\newcommand{\scrO}{\EuScript{O}}

\newcommand{\scrR}{\EuScript{R}}
\newcommand{\scrS}{\EuScript{S}}
\newcommand{\scrT}{\EuScript{T}}
\newcommand{\scrU}{\EuScript{U}}
\newcommand{\scrV}{\EuScript{V}}
\newcommand{\scrW}{\EuScript{W}}
\newcommand{\scrX}{\EuScript{X}}
\newcommand{\scrY}{\EuScript{Y}}
\newcommand{\scrZ}{\EuScript{Z}}
\newcommand{\con}{\operatorname{con}}

\newcommand{\scrIc}{\scrI^{\kern 0.5pt\mathrm{c}}}

\newcommand{\Delt}{\Updelta_{\kern 0.05em 0}}
\newcommand{\DeltAff}{\Updelta_{\kern 0.05em 0}^{\aff}}
\newcommand{\WDelt}{W_{\kern -0.1em \Updelta}}
\newcommand{\WDeltaff}{W_{\kern -0.1em \Updelta_{\aff}}}
\newcommand{\Wkern}[1]{W_{\kern -0.1em #1}\kern 0.05em}
\newcommand{\wo}[1]{w_{\kern -0.075em #1}}
\newcommand{\wop}[1]{w^{\phantom J}_{\kern -0.1em #1}}
\newcommand{\GammaJ}{\Upgamma_{\kern -0.05em J}}
\newcommand{\GammaS}{\Upgamma_{\kern -0.05em \scrJ}}

\newcommand{\xlup}[1]{{}^{#1}\kern -0.15em x}
\newcommand{\xlupmax}[1]{{}^{#1}\kern -0.25em x}
\newcommand{\iDelta}{\iota_{\kern -0.075em \Updelta}}
\newcommand{\Phisub}[1]{\Phi_{\kern -0.1em #1}}
\newcommand{\aff}{\operatorname{\mathsf{aff}}\nolimits}

\newcommand\Curve{\mathrm{C}}

\def\Id{\mathop{\rm{Id}}\nolimits}

\newcommand\chamC{\mathsf{C}}
\newcommand\chamD{\mathsf{D}}
\newcommand\chamE{\mathsf{E}}

\newcommand\stautilt{\operatorname{\mathsf{s}\uptau\mathsf{\textsf{-}tilt}}}
\newcommand{\smc}{\operatorname{\mathsf{smc}}}
\newcommand{\dsmc}{\operatorname{\mathsf{\textsf{-}smc}}}
\newcommand{\sbrick}{\operatorname{\mathsf{sbrick}}}
\newcommand{\silt}{\operatorname{\mathsf{\textsf{-}silt}}}

\newcommand{\Hyp}{\mathsf{H}}

\definecolor{Pink}{RGB}{230 56 243}
\tikzset{
W/.style={circle,draw=black,circle,fill=white,inner sep=0pt, minimum size=4pt},
B/.style={circle,draw=black!80!white,circle,fill=black!80!white,inner sep=0pt, minimum size=4pt},
Or/.style={circle,draw=Orange,circle,fill=Orange,inner sep=0pt, minimum size=4pt},
P/.style={circle,draw=Pink,circle,fill=Pink,inner sep=0pt, minimum size=4pt},
R/.style={circle,draw=black!80!white,circle,fill=red!80!white,inner sep=0pt, minimum size=4pt},  
}

\newcommand{\Dfive}[5]{%
\begin{tikzpicture}[scale=0.21]
\node at (1,0) [#1] {};
\node at (2,0) [#2] {};
\node at (2,1) [#3] {};
\node at (3,0) [#4] {};
\node at (4,0) [#5] {};
\end{tikzpicture}
}

\newcommand{\Dfour}[5]{%
\begin{tikzpicture}[scale=0.21]
\node at (1,0) [#1] {};
\node at (2,0) [#2] {};
\node at (2,1) [#3] {};
\node at (3,0) [#4] {};
\end{tikzpicture}
}

\newcolumntype{C}[1]{>{\centering\let\newline\\\arraybackslash\hspace{0pt}}m{#1}}

\renewcommand{\H}{\mathrm{H}}

\begin{document}

\title{Spherical Objects in Dimensions Two and Three}

\author{Wahei Hara}
\address{W.~Hara,  The Mathematics and Statistics Building, University of Glasgow, University Place, Glasgow, G12 8QQ, UK.}\email{wahei.hara@glasgow.ac.uk}
\author{Michael Wemyss}
\address{M.~Wemyss, The Mathematics and Statistics Building, University of Glasgow, University Place, Glasgow, G12 8QQ, UK.}
\email{michael.wemyss@glasgow.ac.uk}

\begin{abstract} 
This paper classifies spherical objects in various geometric settings in dimension two and three, including both minimal and partial crepant resolutions of Kleinian singularities, as well as arbitrary flopping $3$-fold contractions with only Gorenstein terminal singularities.  The main result is much more general: in each such setting, we prove that all objects in the associated null category $\scrC$ with no negative Ext groups are the image, under the action of an appropriate braid or pure braid group, of some object in the heart of a bounded t-structure.  The corollary is that all objects $x$ which admit no negative Exts, and for which $\Hom_\scrC(x,x)=\mathbb{C}$, are the images of the simples.  A variation on this argument goes further, and classifies all bounded t-structures on~$\scrC$.  There are multiple geometric, topological and algebraic consequences, primarily to autoequivalences and stability conditions. Our main new technique also extends into representation theory, and we establish that in the derived category of a finite dimensional algebra which is silting discrete, every object with no negative Ext groups lies in the heart of a bounded t-structure.  As a consequence, every semibrick complex can be completed to a simple minded collection.
\end{abstract}
\thanks{The authors were supported by EPSRC grant EP/R034826/1, and MW additionally by ERC Consolidator Grant 101001227 (MMiMMa). }
\maketitle{}

\section{Introduction}

In reasonable categories, an object is \emph{spherical} if its self-extension groups behave like the cohomology of a sphere.  These objects have attracted interest across symplectic geometry, algebraic geometry and representation theory, since  each spherical object generates a symmetry in the form of a twist  autoequivalence \cite{ST}. It turns out that these control much of the structure of autoequivalence groups, at least in small dimension.  

The question of classifying spherical objects has been approached by various authors \cite{AS, IU, IUU, ST, BDL, SW}, not least because of the resulting topological and geometric consequences. The viewpoint of this paper is that any such classification \emph{must} necessarily be the consequence of something more general. Spherical objects are not strictly speaking the correct objects in either dimension two \cite[\S10]{IW9} or dimension three \cite{Toda, DW1}, and furthermore birational geometry requires us to work with singular varieties, for which the self-extension groups are usually unbounded.  

Our main new insight is that a more general classification is indeed possible, and we identify the homological condition for which such a result exists.  In the happy situation when spherical objects exist, and are relevant, their classification is a consequence.

\subsection{Motivation and Setting}
This paper is concerned mainly with the following three geometric settings, but later in \S\ref{sec: intro gen} it also works much more generally.

\begin{enumerate}
\item (surfaces) The minimal resolution $\scrZ\to\mathbb{C}^2/\Gamma$ of a Kleinian singularity.  
\item (surfaces) A crepant partial resolution $\scrY\to\mathbb{C}^2/\Gamma$ of a Kleinian singularity.
\item ($3$-folds) A $3$-fold flopping contraction $\scrX\to \Spec \scrR$, where  $\scrX$ has at worst Gorenstein terminal singularities, e.g.\ $\scrX$ is smooth, and $\scrR$ is complete local.
\end{enumerate}
The case when $\scrZ\to\mathbb{C}^2/\Gamma$ is the minimal resolution has attracted the most attention \cite{IU, BDL}. That situation is unnaturally easy for two separate reasons.  First, there are very few examples: just two infinite families and three sporadic cases.  Second, the category $\scrC$ below is then controlled by an intrinsically formal DG-algebra \cite{ST}, which makes many computations easier, and in particular shows that $\scrC$ is equivalent to many other categories in the literature.  In the other geometric settings we should expect neither, as the controlling DG-algebra is usually unbounded, and is far from formal.

Nonetheless, for any $X=\scrX, \scrY$ or $\scrZ$ above, consider 
\begin{align}
\scrC&\colonequals \{ x\in\Db(\coh X)\mid \mathbf{R
}f_*x=0\} \subset \Db(\coh X).\label{eqn: define C intro}
\end{align}
The category $\scrC$ has finite dimensional Hom spaces, and further is $d$-CY, where $d=\dim X$, if $X$ is smooth.  Our interest will be in spherical objects, namely those $x\in\scrC$ for which
\[
\Hom_\scrC(x,x[i])=
\begin{cases}
\mathbb{C}&\mbox{if }i=0,d\\
0&\mbox{else.}
\end{cases}
\]
In the case of the minimal resolution $\scrZ\to\mathbb{C}^2/\Gamma$, and also for some very special smooth flopping contractions, these objects are intimately related to generation questions regarding the autoequivalence group of $\scrC$, and to the existence of Lagrangians in plumbings of spheres.

\medskip
However, in general, autoequivalences are controlled by more complicated objects \cite{Toda, DW1, Booth}, and what makes these objects interesting is that they are \emph{not} determined by their self-extension groups, even in the lucky case when those groups are bounded.  This fact has, until now, hampered any attempt to classify them, since it is not even clear what we are supposed to classify.

\medskip
The main results below show that, rather remarkably, the only conditions that matter are that $\Hom_{\scrC}( x, x[i])=0$ for all $i<0$, and $\Hom_{\scrC}( x, x)=\mathbb{C}$. These would seem to be rather weak conditions, however we will prove that these objects are precisely the orbits of the simples under the natural action of the mutation functors, even in very singular settings.  In particular, there are actually very few of them.   

In fact, and much more generally, it is possible to completely characterise those objects which satisfy only the condition $\Hom_{\scrC}( x, x[i])=0$ for all $i<0$.  These, it turns out, are precisely the objects in the orbits of the standard t-structure under the action of the mutation functors.  Furthermore, all bounded t-structures on $\scrC$ can be similarly classified.  From these results, and others, we extract geometric, topological and algebraic corollaries.

\subsection{Main Results}
For partial resolutions $\scrY$, or for $3$-fold flopping contractions $\scrX$, it is the lack of uniqueness that makes these settings complicated, not the fact that they are singular. Whilst there is only one minimal resolution $\scrZ$, there are plenty of partial resolutions $\scrY$, and furthermore the lack of uniqueness is \emph{the} defining feature of the $3$-fold flops setting.  To compensate for this requires us to work, at least initially, in groupoids.  

\medskip
As is well-known, and is recalled in \S\ref{subsec:arrangements}, to each $\scrX$, $\scrY$ or $\scrZ$ is an associated finite simplicial hyperplane arrangement $\scrH$.  In the case of the minimal resolution, this is simply the associated ADE root system, and in general the construction is roughly similar.  To each chamber $\chamC$ in $\scrH$ we assign a category $\scrC_\chamC$, defined in \S\ref{sec: bij MMP} in terms of an associated noncommutative resolution (or variant). To each path in the Deligne groupoid $\upbeta\colon\chamC\to\chamD$, we assign a composition of mutation functors $\Phi_\upbeta\colon\scrC_\chamC\to\scrC_\chamD$.

\begin{thm}[\ref{make module}]
If $x\in\scrC=\scrC_\chamC$ satisfies $\Hom_{\scrC}( x, x[i])=0$ for all $i<0$, then there exists $\upbeta\colon\chamC\to\chamD$  such that $\Phi_\upbeta (x)\cong y$, where $y$ is an object in homological degree zero.
\end{thm}

The proof is both short and elementary, and there are only two key points.  First, cohomology spread should be measured with respect to the noncommutative (equivalently, perverse) t-structure, not with respect to coherent sheaves.  Second, if say $x$ has cohomology lying within some bounded region $[a,b]$, then the proof finds some $\Phi_\upalpha(x)$ whose cohomology lies in a strictly smaller region $[a+1,b]$.  This is achieved by using the fact that any finite poset has a maximal element.

The first corollary is the following. To set notation, write $\scrS_1,\hdots,\scrS_n$ for the simple modules in the heart of the standard algebraic t-structure (see \S\ref{subsec: simples}).

\begin{cor}[\ref{cor 2 3folds}]\label{cor 2 3folds intro}
Suppose that $x\in\scrC_\chamC$ satisfies $\Hom( x, x[i])=0$ for all $i<0$.
\begin{enumerate}
\item If $\dim_{\mathbb{C}}\Hom(x_i, x_j)\cong \mathbb{C}\updelta_{ij}$ for all indecomposable summands $x_i, x_j$ of $x$, then there exists a subset $\scrI\subseteq\{1,\hdots,n\}$ and a path $\upgamma\colon \chamC\to\mathsf{E}$ such that $\Phi_\upgamma(x)\cong \bigoplus_{i\in \scrI}\scrS_{i}$.
\item If $\dim_{\mathbb{C}}\Hom(x, x)\cong \mathbb{C}$, then there exists $\upgamma\colon \chamC\to\mathsf{E}$ and $i$ such that $\Phi_\upgamma(x)\cong \scrS_i$.
\end{enumerate}
\end{cor}
Given the control of the mutation functor orbit obtained in \ref{cor 2 3folds intro}, a standard variation on results of \cite{IUU} and \cite{BDL} can then be used to show that the space of stability conditions on $\scrC$ is connected and thus $\cStab{}\scrC=\Stab\scrC$.   However, a much stronger statement holds. 

The following result classifies all bounded t-structures on $\scrC$ in terms of the standard hearts. These are, by definition, the extension closures of the simples $\scrS_1,\hdots,\scrS_n$ (see \S\textnormal{\ref{subsec: cont alg}}).   This is a purely homological result, which is new in all cases, and its proof does not require stability techniques.

\begin{cor}[\ref{cor: no exotic}]
If $\scrH$ is the heart of a bounded t-structure on $\scrC_\chamC$, then there exists $\Phi_\upbeta$ for some $\upbeta\colon \chamC\to\chamD$ such that $\Phi_\upbeta(\scrH)$ is the standard heart on $\scrC_\chamD$.  In particular, $\scrH$ is a finite length category, with finitely many simples.
\end{cor}
From this, in \ref{cor: stability connected} the statement $\cStab{}\scrC=\Stab\scrC$ immediately follows, and avoids the use of any fancy deformation argument.  In \ref{cor: FM char} we also give an intrinsic characterisation of the group $\Auteq\scrC$ in terms of certain FM transforms. 

\subsection{Geometric Corollaries}
By removing the noncommutative resolutions technology implicit in the previous subsection, in particular the use of the mutation functors, it is possible to translate the above results into more standard geometric language.  We translate some results here, with more details and results given in \S\ref{sec4: geo cors}.

In the case of the minimal resolution $\scrZ$, the mutation functors are functorially isomorphic to the Seidel--Thomas twist functors.  Recall that $\scrC$ is the category geometrically defined in \eqref{eqn: define C intro}, and to each exceptional curve $\Curve_i$ in $\scrZ$, the associated $\scrO_{\Curve_i}(-1)$ is spherical, so gives rise to a spherical twist functor $t_i$.  These generate a subgroup $\Br$ of $\Auteq\scrC$.

Part (2) of the following recovers the main result of \cite{BDL}.

\begin{cor}[\ref{thm: minimal main}]\label{intro: minimal main}
Consider $\scrZ\to\mathbb{C}^2/\Gamma$, and let $x\in\scrC$.  Then the following hold. 
\begin{enumerate}
\item If $\Hom_\scrC(x,x[j])=0$ for all $j<0$, then there exists $T\in\Br$ such that $T(x)$ is a concentrated in homological degree zero.
\item Every spherical object in $\scrC$ belongs to the orbit, under the action of the braid group, of the objects $\scrO_{\Curve_1}(-1),\hdots,\scrO_{\Curve_n}(-1)$.
\end{enumerate}
Furthermore, the heart of every bounded t-structure on $\scrC$ is the image, under the action of the group $\Br$, of the module category of the preprojective algebra of (finite) ADE type.
\end{cor}

The case of partial resolutions $\scrY\to\mathbb{C}^2/\Gamma$ is similar, but mildly more difficult to state, so details are left to \ref{cor: cpartial groupoid main}.   

For $3$-fold flopping contractions $\scrX\to\Spec\scrR$, the mutation functor $\Phi_i$ is functorially isomorphic to the inverse of the Bridgeland--Chen flop functor \cite[4.2]{HomMMP}, which will be written $\Flop_i$.  These are the square roots of the spherical twist functors, suitably interpreted (see \ref{rem: calibration}).  The following generalises the particular example in \cite[6.12(1)]{SW} to all $3$-fold flops.

\begin{cor}[\ref{cor: flop main}]\label{intro: flop main}
Let $\scrX\to\Spec\scrR$ be a $3$-fold flopping contraction, where $\scrX$ has at worst Gorenstein terminal singularities, and consider $x\in\scrC$.
\begin{enumerate}
\item The following statements are equivalent.
\begin{enumerate}
\item $\Hom_\scrC(x,x[j])=0$ for all $j<0$.
\item There exists a sequence of flop functors such that $\Flop_\upbeta(x)$ belongs to perverse sheaves, on a possibly different $\scrX^+\to\Spec\scrR$.
\end{enumerate}
\item  Furthermore, the following are equivalent.
\begin{enumerate}
\item $\Hom_\scrC(x,x[j])=0$ for all $j<0$ and $\Hom_\scrC(x,x)\cong\mathbb{C}$.
\item There exists $\upbeta$ such that $\Flop_\upbeta(x)\cong\scrO_{\Curve_i}(-1)$ for some curve $\Curve_i$ on a possibly different $\scrX^+\to\Spec\scrR$.
\end{enumerate}
\end{enumerate}
Furthermore, the heart of every bounded t-structure on $\scrC$ is the image, under the action of the group generated by the flop functors, of the module category of the contraction algebra of some $\scrX^+\to\Spec\scrR$ obtained from $\scrX\to\Spec\scrR$ by iterated flop.
\end{cor}
In particular, when $\scrX$ is smooth, $x\in\scrC$ is spherical if and only if there exists a $(-1,-1)$-curve $\Curve_i$, on a possibly different $\scrX^+$, and a composition $\upbeta\colon\chamD\to\chamC$ in the Deligne groupoid such that $\Flop_\upbeta(\scrO_{\Curve_i}(-1))\cong x$.  Perhaps of more interest is when the groupoid result \ref{intro: flop main} is used to prove statements regarding the orbit of the pure braid group $\PBr$ on a fixed $\scrC$. In \ref{cor: flop main PBr} we describe the spherical objects in terms of the orbit under this action, generalising \cite[6.12(2)]{SW} to all $3$-fold flops.

\subsection{Algebraic Corollaries}

A more algebraic corollary of \ref{cor 2 3folds intro} is a generalisation of a result of Crawley--Boevey \cite[Lemma 1]{BCB} which states that every brick for the preprojective algebra $\Uppi$ of type ADE has dimension vector equal to a root.  The categories $\scrC$ are either controlled by algebras of the form $e_\scrI\Uppi e_\scrI$, or by contraction algebras, depending on the setting. The underlying combinatorial structure of these are \emph{restricted roots}, recalled in \S\ref{subsec:arrangements} and illustrated in \ref{ex: Dynkin res}.

\begin{cor}[\ref{ePie dim vectors}]
Consider the projective algebra $\Uppi$ of type \textnormal{ADE}, and $\scrI\subseteq\Updelta$.  Then every brick in $\fmod e_\scrI\Uppi e_\scrI$ has dimension vector equal to a primitive restricted root. 
\end{cor}
A version of the above is proved in \ref{cor contraction brick dim} for contraction algebras $\Lambda_{\con}$, which are recalled in \S\ref{subsec: cont alg}.  These algebras control much of the birational geometry of $3$-folds, and their representation theory is usually wild.

\subsection{Silting Discrete Derived Categories}\label{sec: intro gen}
The above techniques extend further.  Recall that a finite dimensional algebra $A$ is called silting discrete if the category $\Kb(\proj A)$ is a silting discrete triangulated category (see \S\ref{sec: silting prelims}).

This class of silting discrete finite dimensional algebras is surprisingly rich: it contains all local algebras, Erdmann's algebras of dihedral, semidihedral and quaternion type, various preprojective algebras and Brauer graph algebras, together with more obvious representation-finite examples (see e.g.\ \cite[p1]{AH}).  From the viewpoint of this paper the important point is that the full list of examples also includes all $3$-fold contraction algebras $\Lambda_{\con}$. The resulting categories $\Db(\Lambda_{\con})$ are related to, but are very different from, the categories $\scrC$ in the previous sections.  

The analogy between $\scrC$ and $\Db(\Lambda_{\con})$ allows us to push the previous techniques into the general silting discrete setting. The following is new in all cases.

\begin{thm}[\ref{thm: silting main}]\label{thm: silting main intro}
If $A$ is silting discrete, $\scrT=\Db(A)$ and $x\in\scrT$, then the following statements are equivalent.
\begin{enumerate}
\item $\Hom_\scrT(x,x[i])=0$ for all $i<0$.
\item $x$ belongs to the heart of a bounded t-structure.
\end{enumerate}
\end{thm}

In this level of generality, there are fewer autoequivalences, so the strategy for the proof of \ref{thm: silting main intro} minimises the cohomology with respect to the bounded t-structures on the fixed~$\scrT$. This avoids groupoids.  The object $x$ has cohomology lying within some bounded region $[a,b]$ with respect to the standard t-structure.  The proof first finds a new t-structure whose cohomology functors applied to $x$ lie in a strictly smaller region $[a+1,b]$.  A simple induction argument then shows that there is a bounded t-structure such that $x[b]$ belongs to its heart.

The following are immediate consequences, which may be of independent interest.  The definition of brick and semibrick complexes is recalled in \ref{dfn: semibrick complex}.

\begin{cor}[\ref{cor: silting simples}]\label{thm: silting simples main intro}
Let $A$ be silting discrete, $\scrT\colonequals\Db(A)$ and $x\in\scrT$.
\begin{enumerate}
\item $x$ is a sum of simples in the heart of a bounded t-structure iff $x$ is a semibrick.
\item $x$ is a simple in the heart of a bounded t-structure iff $x$ is a brick.
\end{enumerate}
In particular, if $x\in\scrT$ is a semibrick complex, then there exists a simple minded collection~$\scrU$ such that $x$ is contained in $\scrU$. 
\end{cor}
It follows in \ref{cor: semi of maximal rank} that for silting discrete algebras, every semibrick pair forms a subset of a simple minded collection, and further that every semibrick pair of maximal rank is itself a simple minded collection.  The slightly subtle point is that in general the simple minded collection containing the semibrick pair need not be 2-term with respect to the standard t-structure, which explains why \ref{thm: silting simples main intro} does not contradict \cite{HI, BH}.  We give some explicit examples in \ref{ex: no contra} and \ref{ex: no contra 2} which illustrate this phenomenon.

\medskip
The above general results apply to $3$-fold contraction algebras $\Lambda_{\con}$, where more can be said.  The autoequivalence group of $\Db(\Lambda_{\con})$ is large, and so the above simplifies. We show in \ref{cor: semibrick for cont alg} that for contraction algebras the brick and semibrick complexes can be characterised as orbits of the brick modules under the action of the autoequivalence group.  The result further improves if we instead pass to the groupoid, but details are left to \S\ref{subsec contraction cor}.

\subsection*{Acknowledgements}
The authors thank David Pauksztello for generously sharing his many insights into simple minded collections in silting discrete triangulated categories, and additionally for passionately communicating his love of spectral sequences.  M.W.\ also thanks Aidan Schofield, Ailsa Keating, Ivan Smith, Anthony Licata, Asilata Bapat, Anand Deopurkar and Matt Pressland for very helpful conversations, and ICERM for hosting the conference `Braids in Symplectic and Algebraic Geometry' in March 2022.

\subsection*{Conventions}
To be consistent between the algebraic and geometric settings considered, all objects will be taken over $\mathbb{C}$. Furthermore, since e.g.\ $\Hom(\H^i(\Phi_j(y)),\H^i(x))$ is slightly painful, brackets will often be dropped when applying functors to objects. Doing this, the above simplifies to $\Hom(\H^i(\Phi_jy),\H^ix)$.

\section{Preliminaries}

This section briefly recalls the known properties of the category $\scrC$, and sets notation.  This is largely a summary of \cite{HW2, IW9} suitable for our purposes. 

\subsection{Noncommutative Resolutions and Variants}\label{subsec: NC res} 
Since by definition $\scrC$ is a full subcategory of $\Db(\coh X)$, it is possible to take cohomology of objects in $\scrC$ with respect to the standard t-structure on $\Db(\coh X)$.  These cohomology sheaves were used heavily e.g.\ in \cite{IU}.  However, as in \cite[\S6]{SW} it turns out to be easier to control the cohomology with respect to a different t-structure, namely perverse sheaves. 

Throughout this paper we will abuse notation, and write $f$ for either the minimal resolution $ \scrZ\to\mathbb{C}^2/\Gamma$, a fixed partial crepant resolution $\scrY\to\mathbb{C}^2/\Gamma$, or a fixed $3$-fold flopping contraction $\scrX\to\Spec\scrR$ where $\scrX$ has only Gorenstein terminal singularities and $\scrR$ is complete local.  It is well known that the reduced fibre above the unique singular point is a tree of curves, with each irreducible component $\Curve_i\cong\mathbb{P}^1$. Furthermore, $X$ will denote the choice of fixed $\scrZ$, $\scrY$ or $\scrX$, and from that choice the category $\scrC$ is defined to be
\[
\scrC= \{ a\in\Db(\coh X)\mid \mathbf{R}f_*a=0\}.
\]
By \cite{VdB} there exists a tilting bundle $\scrO\oplus \scrN$ on $X$, where $\scrN^*$ is generated by global sections, which induces an equivalence of categories
\begin{equation}
\Psi_X\colon\Db(\coh X)\xrightarrow{\RHom_X(\scrO\oplus\scrN,-)}\Db(\fmod\Lambda)\label{NCCR Db}
\end{equation}
where $\Lambda=\End_X(\scrO\oplus \scrN)$.  The algebra $\Lambda$ admits an idempotent $e$ corresponding to the summand $\scrO$, and across the equivalence \eqref{NCCR Db},  $\scrC$ corresponds to those complexes whose cohomology groups are all annihilated by $e$, or in other words, those complexes whose cohomology groups are $\Lambda/(e)$-modules.

\subsection{Intersection Arrangements}\label{subsec:arrangements}
To any $X$ equal to $\scrZ$, $\scrY$ or $\scrX$ in \S\ref{subsec: NC res}, it is possible to associate Dynkin data in the form of a subset $\scrI\subset \Updelta$ with $\Updelta$ ADE. 
\begin{enumerate}
\item For the minimal resolution $\scrZ\to\mathbb{C}^2/\Gamma$, $\Updelta$ is the ADE Dynkin diagram corresponding to $\Gamma$ via McKay Correspondence, and $\scrI=\emptyset$. 
\item For a partial crepant resolution $\scrY\to\mathbb{C}^2/\Gamma$, $\Updelta$ is the ADE Dynkin diagram corresponding to $\Gamma$ via McKay Correspondence, and $\scrI$ is the subset of curves that are contracted by the morphism $\scrZ\to\scrY$. 
\item For a flopping contraction $\scrX\to\Spec\scrR$, a generic hyperplane section slices to a partial crepant resolution as in (2) above \cite{Pagoda}, which thus associates some $\scrI\subset\Updelta$.
\end{enumerate}
\begin{rem}\label{rem: subset conventions}
The exceptional curves in $\scrZ$ will be written $\Curve_i$ with $i\in\Updelta$. 
For partial resolutions, the subset $\scrI$ is the choice of curves that get contracted from $\scrZ$.  We will abuse notation, and write $\Curve_i$ with $i\in\scrIc=\Updelta\setminus \scrI$ for the exceptional curves that survive, namely the curves in $\scrY$.  Further, we will also write $\Curve_i$ with $i\in\scrIc$ for the exceptional curves in $\scrX$.
\end{rem}
The ADE diagram $\Updelta$ has an associated Cartan $\mathfrak{h}=\bigoplus\mathbb{R}\upalpha_i$, where $\upalpha_i$ are the simple roots. For any $\scrI\subset\Updelta$, define $\mathfrak{h}_{\scrI}$ to be the quotient of $\mathfrak{h}$ by the subspace spanned by $\{ \upalpha_i \mid i \in \scrI \}$.  The  associated quotient map will be written
 \[
 \uppi_{\scrI} \colon \mathfrak{h} \to \mathfrak{h}_{\scrI},
 \] 
and note that $\mathfrak{h}_{\scrI}$ has basis $\{ \uppi_{\scrI}(\upalpha_i) \mid i \notin \scrI\}=\{ \uppi_{\scrI}(\upalpha_j)\mid j\in \scrIc\}$. The \emph{restricted positive roots} in $\mathfrak{h}_\scrI$ are precisely the non-zero images of positive roots under $\uppi_{\scrI}$.  A restricted positive root is \emph{primitive} if it is not a multiple of another positive restricted root.

\begin{example}\label{ex: Dynkin res}
Consider the Dynkin data $\,\Dfive{B}{P}{B}{P}{B}\,$, where $\scrI$ is the set of black nodes.  Projecting all twenty positive roots of $D_5$ via $\uppi_\scrI$, and discounting the zeros, gives the set
\[
\{  10, 01, 11, 21, 22 \}
\] 
where for example $11$ is shorthand for the coordinate $(1,1)$ in $\mathfrak{h}_\scrI$ under the prescribed basis above. These are the restricted roots.  The primitive restricted roots are $\{  10, 01, 11, 21 \}$.
\end{example}

Now write $\Uptheta_{\scrI}$ for the dual of the real vector space $\mathfrak{h}_\scrI$.  For each restricted positive root $0\neq\upbeta=\uppi_\scrI(\upalpha)\in \mathfrak{h}_\scrI$ consider the dual hyperplane
\[ 
\Hyp_\upbeta\colonequals \{ (\upvartheta_i) \mid \textstyle\sum\upbeta_i\upvartheta_i=0 \} \subseteq \Uptheta_{\scrI}. 
\]
Since there are only finitely many restricted roots, the collection of $\Hyp_\upbeta$ forms a finite hyperplane arrangement in $\Uptheta_{\scrI}$, which we refer to as the \emph{intersection arrangement}.  In general, this need not be Coxeter.  

\begin{example}\label{ex: Dynkin res 2}
Continuing \ref{ex: Dynkin res}, for example the restricted root $21$ gives the dual hyperplane $2x+y=0$.  The full intersection arrangement is the following.
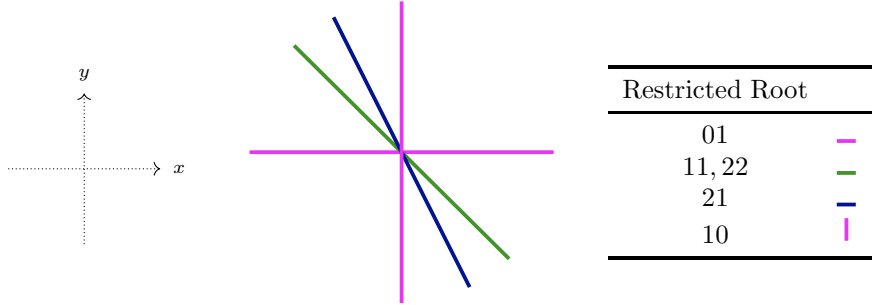
\[
\begin{array}{cccc}
\begin{array}{c}
\begin{tikzpicture}[scale=0.5]
\draw[->,densely dotted] (180:2cm)--(0:2cm);
\node at (0:2.5) {$\scriptstyle x$};
\draw[->,densely dotted] (-90:2cm)--(90:2cm);
\node at (90:2.5) {$\scriptstyle y$};
\end{tikzpicture}
\end{array}
&
\begin{array}{c}
\begin{tikzpicture}[scale=1]
\draw[line width=\mythick mm,Pink] (180:2cm)--(0:2cm);
\draw[line width=\mythick mm,Green] (135:2cm)--(-45:2cm);
\draw[line width=\mythick mm, Blue] (116.57:2cm)--(-63.43:2cm);
\draw[line width=\mythick mm,Pink] (90:2cm)--(-90:2cm);
\end{tikzpicture}
\end{array}&
\begin{array}{c}
\begin{tabular}{ccc}
\toprule
Restricted Root&\\
\midrule
$01$&$\tikz\draw[line width=\mythick mm, Pink] (0,0) -- (0.25,0);$\\
$11, 22$&$\tikz\draw[line width=\mythick mm, Green] (0,0) -- (0.25,0);$\\
$21$&$\tikz\draw[line width=\mythick mm, Blue] (0,0) -- (0.25,0);$\\
$10$&$\tikz\draw[line width=\mythick mm, Pink] (0,-0.15) -- (0,0.15);$\\
\bottomrule
\end{tabular}
\end{array}
\end{array}
\]
\end{example}

\subsection{Deligne Groupoid}\label{sec: Deligne}

Write $\scrH$ for the arrangement obtained from the Dynkin data $\scrI\subset\Updelta$ as defined in \S\ref{subsec:arrangements}, which is a finite simplicial hyperplane arrangement.  This subsection summarises well-known constructions associated to $\scrH$.

\begin{dfn} \label{def:Gamma graph}
Consider the graph $\Gamma_{\scrH}$ of oriented arrows, which has as vertices the chambers (i.e.\ the connected components) of $\bR^n\backslash\scrH$, and there is a unique arrow $a \colon  \chamC\to \chamD$ from chamber $\chamC$ to chamber $\chamD$ if the chambers are adjacent, otherwise there is no arrow. For an arrow $a\colon  \chamC\to \chamD$, write $s(a)\colonequals  \chamC$ and $t(a)\colonequals  \chamD$.
\end{dfn}
By definition, if there is an arrow $a \colon  \chamC\to \chamD$, there is a unique arrow $b\colon  \chamD\to \chamC$. 

\begin{dfn}
Given $\chamC$, then a chamber $\mathsf{op}(\chamC)$ is said to be {\it opposite} $\chamC$ if there is a line in $\mathbb{R}^n$ passing through $\chamC$, $\mathsf{op}(\chamC)$, and the origin. The opposite chamber $\mathsf{op}(\chamC)$ is unique.
\end{dfn}

As for quivers, a \emph{positive path of length~$n$} in $\Gamma_{\scrH}$ is defined to be a formal symbol
\[
p=a_n\circ \hdots\circ a_2\circ a_1,
\]
 whenever there exists a sequence of vertices $v_0,\hdots,v_n$ of $\Gamma_{\scrH}$ and arrows $a_i\colon v_{i-1}\to v_i$ in $\Gamma_{\scrH}$. Set $s(p)\colonequals  v_0$ and $t(p)\colonequals  v_n$, and write $p\colon s(p)\to t(p)$. The notation $\circ$ is composition, but we will often drop the $\circ$'s in future.  If 
$q=b_m\circ\hdots\circ b_2 \circ b_1$ is another positive path with $t(p)=s(q)$, consider the formal symbol
\[
q\circ p\colonequals  b_m\circ\hdots\circ b_2 \circ b_1
\circ
a_n\circ \hdots\circ a_2\circ a_1,
\]
and call it the composition of $p$ and $q$.

\begin{dfn}
A positive path is called \emph{reduced} if it does not cross any hyperplane twice. 
\end{dfn}
Since $\scrH$ is finite, reduced positive paths coincide with shortest positive paths \cite[4.2]{Paris}, in the sense that there is no positive path in $\Gamma_{\scrH}$ of smaller length, with the same endpoints.

\begin{exa}\label{ex: 6 cham 1}
Consider the following hyperplane arrangement $\scrH$ with $6$ chambers, and with chamber $\chamC$ marked.  Then $\mathsf{op}(\chamC)$, and all reduced paths leaving $\chamC$, are illustrated below.
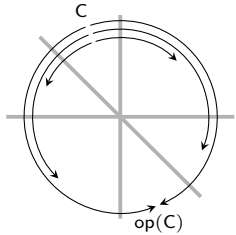
\[
\begin{tikzpicture}[scale=0.75,bend angle=15, looseness=1,>=stealth]
\coordinate (A1) at (135:2cm);
\coordinate (A2) at (-45:2cm);
\coordinate (B1) at (153.435:2cm);
\coordinate (B2) at (-26.565:2cm);
\draw[line width=\mythick mm,black!30] (A1) -- (A2);
\draw[line width=\mythick mm,black!30] (-2,0)--(2,0);
\draw[line width=\mythick mm,black!30] (0,-1.8)--(0,1.8);
\draw[->] ([shift=(108:1.4cm)]0,0) arc (108:45:1.4cm);
\draw[->] ([shift=(108:1.55cm)]0,0) arc (108:-22.5:1.55cm);
\draw[->] ([shift=(108:1.7cm)]0,0) arc (108:-66:1.7cm);
\draw[->] ([shift=(111:1.4cm)]0,0) arc (111:157.5:1.4cm);
\draw[->] ([shift=(111:1.55cm)]0,0) arc (111.5:225:1.55cm);
\draw[->] ([shift=(111:1.7cm)]0,0) arc (111.5:291:1.7cm);

\node at (110:2cm) {$\scriptstyle \chamC$};
\node at (-70:2cm) {$\scriptstyle \mathsf{op}(\chamC)$};
\end{tikzpicture}
\]
\end{exa}

As in \cite[p7]{Delucchi}, let $\sim$ denote the smallest equivalence relation, compatible with composition, that identifies all morphisms that arise as positive reduced paths with the same source and target. Write $\mathrm{Free}(\Gamma_\scrH)$ for the free category on the graph $\Gamma_\scrH$, where morphisms are directed paths.  The quotient category 
\[
\scrG^{+}_\scrH \colonequals \mathrm{Free}(\Gamma_\scrH)/\sim, 
\]
is called the category of positive paths. 

\begin{dfn}
The elements of $\scrG^{+}_\scrH$ which are reduced are called \emph{atoms}.  That is, atoms are the (now unique) shortest positive paths between chambers.  
\end{dfn}

The set of atoms beginning from a fixed chamber admits a partial order.  
\begin{dfn}\label{dfn: atoms order}
If $\upalpha,\upbeta$ are atoms that both begin at $\chamC$, write $\upalpha\geq \upbeta$ if there exists an atom $\upgamma$ such that $\upalpha=\upgamma\circ\upbeta$. 
\end{dfn}
The identity is minimal, and any atom $\ell\colon\chamC\to\mathsf{op}(\chamC)$ is the maximal element.  This convention is chosen to match the Bruhat order, but has the rather unfortunate consequence of being opposite to the order on tilting and silting complexes. This explains why maximal elements in \S\ref{sec: classification} switch to minimal elements in \S\ref{sec: silting section}.

\begin{dfn} \label{def: Deligne groupoid completion}
The \emph{arrangement (=Deligne) groupoid}  $\scrG_{\scrH}$ is the groupoid defined as the groupoid completion of  $\scrG_{\scrH}^{+}$, that is, a formal inverse is added for every morphism in $\scrG_{\scrH}^{+}$. 
\end{dfn} 

Paths in the groupoid $\scrG_{\scrH}$ can be positive, or negative, or mixed.

\subsection{Categorical Representations and HomMMP}\label{sec: bij MMP}
In each of the settings $X= \scrX, \scrY$ or $\scrZ$ of \S\ref{subsec: NC res}, it is possible to produce a categorical representation of the Deligne groupoid by associating categories to each chamber, and certain equivalences to the wall crossings.  This subsection summarises this construction, mainly to set notation.

Below consider the preprojective algebra $\Uppi$ of an extended ADE Dynkin diagram $\Updelta^{\mathsf{aff}}$.  By convention the vertices will be labelled $i=0,1,\hdots,n$, with $0$ being the extended vertex, and thus the vertex idempotents are labelled $e_0,\hdots,e_n$. For a subset $\scrI\subset\Updelta\subset\Updelta^{\mathsf{aff}}$,  to match the notation in \cite{IW9} set
\begin{equation}
e_\scrI=1-\sum_{i\in\scrI}e_i\label{eqn: idempotent convention}
\end{equation}
and consider the associated contracted preprojective algebra $e_\scrI \Uppi e_\scrI$.  The notation creates one ambiguity, namely that $e_{\{ i\}}$ and $e_i$ are different, since $e_{\{i\}}=1-e_i$. 

To each geometric setting (1)--(3), \S\ref{subsec:arrangements} assigns  a subset $\scrI\subset\Updelta$ and thus a corresponding intersection arrangement $\scrH$. 

\begin{enumerate}
\item For the minimal resolution $\scrZ\to\mathbb{C}^2/\Gamma$, consider the preprojective algebra $\Uppi$ associated to $\Updelta^{\mathsf{aff}}$.  To every chamber $\chamC$ assign the following subcategory of $\Db(\fmod\Uppi)$
\[
\scrC_\chamC\colonequals\{ a\in\Db(\fmod\Uppi)\mid e_0\,\H^*(a)=0 \}
\]
where as above $e_0$ is the idempotent corresponding to the extended vertex.
Note that $\scrC_\chamC$ does not depend on $\chamC$. To each simple wall crossing $s_i\colon\chamC\to\chamD$ with $i=1,\hdots,n$, is assigned the mutation functor $\Phi_i\cong\RHom_\Uppi(\mathrm{I}_i,-)$ where $\mathrm{I}_i$ is the bimodule kernel of the map $\Uppi\to\Uppi/(e_i)$.

\item For a partial crepant resolution $\scrY\to\mathbb{C}^2/\Gamma$, again consider the preprojective algebra $\Uppi$ associated to $\Updelta^{\mathsf{aff}}$.  Now chambers in $\scrH$ are indexed by certain pairs $(x,\scrJ)$ where $x$ belongs to the Weyl group of $\Updelta$, and $\scrJ\subset\Updelta$ \cite[1.12]{IW9}.  To a chamber $\chamC=(x,\scrJ)$  assign the subcategory of $\Db(\fmod e_\scrJ\Uppi e_\scrJ)$ defined to be
\[
\scrC_\chamC\colonequals\{ a\in\Db(\fmod e_\scrJ\Uppi e_\scrJ)\mid e_0\,\H^*(a)=0 \}
\]
where again $e_0$ is the idempotent corresponding to the extended vertex.  To each simple wall crossing $\upomega_i\colon\chamC\to\chamD$ is an associated mutation functor $\Phi_i$, defined generally in \cite[\S6]{IW1}, and in the setting here in \cite[\S5.6]{IW9}.

\item For a flopping contraction $\scrX\to\Spec\scrR$, by \cite{HomMMP, IW9} there is a bijection $\chamC\mapsto M_\chamC$ between chambers in $\scrH$ and certain rigid objects in the category of Cohen--Macaulay modules $\CM\scrR$.  Necessarily $M_\chamC$ has $\scrR$ as a summand. Thus to a chamber $\chamC$ set $\Lambda_\chamC\colonequals \End_\scrR(M_\chamC)$, and assign the subcategory 
\[
\scrC_\chamC\colonequals\{ a\in\Db(\fmod\Lambda_\chamC)\mid e\,\H^*(a)=0 \}
\]
where $e$ is the idempotent corresponding to the summand $\scrR$.
To each simple wall crossing $s_i\colon\chamC\to\chamD$, assign the mutation functor $\Phi_i\cong \RHom_{\Lambda_\chamC}(\Hom_\scrR(M_\chamC,M_\chamD),-)$.

\end{enumerate}

In each case (1)--(3) above, the category $\scrC_\chamC$ corresponds to the more geometric $\scrC$ defined in \eqref{eqn: define C intro} across the derived equivalence $\Uppsi$ in \eqref{NCCR Db}.  The above assignments generate a groupoid $\mathds{G}$, since all functors are equivalences, and in all cases (1)--(3) there is a homomorphism
\begin{equation}
\scrG_{\scrH}\to\mathds{G}.\label{eqn: Deligne rep}
\end{equation}
In the case of the minimal resolution, and in the notation here, this homomorphism was established in \cite[6.6]{IR} (but really first in \cite{ST}), for partial resolutions this is \cite[5.30]{IW9}, and for $3$-fold flopping contractions the homomorphism is \cite[3.22]{DW3}.  
\begin{nota}\label{Phi alpha}
For $\upalpha\in\scrG_{\scrH}$, write $\Phi_\upalpha$ for the image of $\upalpha$ under \eqref{eqn: Deligne rep}.
\end{nota}
From \eqref{eqn: Deligne rep}, it is immediate that there is an induced homomorphism from the vertex group of the Deligne groupoid to $\Auteq\scrC_\chamC$ for any fixed $\scrC_\chamC$.  Since the vertex group is always $\uppi_1(\mathbb{C}^n\backslash \scrH_\mathbb{C})$ \cite{Deligne, Paris,  Salvetti}, it follows that there is an induced homomorphism
\begin{equation}
\uppi_1(\mathbb{C}^n\backslash \scrH_\mathbb{C})\to\Auteq\scrC_\chamC.\label{eqn: Deligne ver}
\end{equation}
This is known as the pure braid group action.  Write $\PBr$ for the image of this homomophism.

\begin{exa}\label{ex: 6 cham 2}
Continuing \ref{ex: 6 cham 1}, label the six chambers as in the left hand side of the following diagram.  The right hand side contains the categorical representation, where each chamber has been replaced by a category, and each wall crossing by a mutation functor.
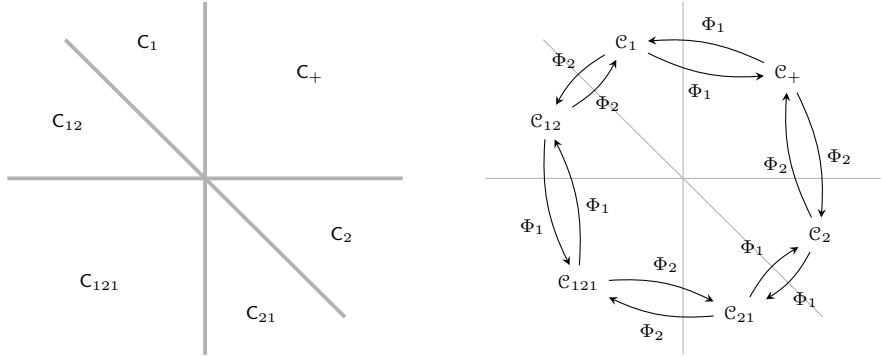
\[
\begin{array}{ccc}
\begin{array}{c}
\begin{tikzpicture}[scale=1.3,bend angle=15, looseness=1,>=stealth]
\coordinate (A1) at (135:2cm);
\coordinate (A2) at (-45:2cm);
\coordinate (B1) at (153.435:2cm);
\coordinate (B2) at (-26.565:2cm);
\draw[line width=\mythick mm,black!30] (A1) -- (A2);
\draw[line width=\mythick mm,black!30] (-2,0)--(2,0);
\draw[line width=\mythick mm,black!30] (0,-1.8)--(0,1.8);
\node (C+) at (45:1.5cm) {$\scriptstyle \chamC_+$};
\node (C1) at (112.5:1.5cm) {$\scriptstyle\chamC_1$};
\node (C2) at (157.5:1.5cm){$\scriptstyle\chamC_{12}$}; 
\node (C-) at (225:1.5cm) {$\scriptstyle\chamC_{121}$}; 
\node (C4) at (-67.5:1.5cm) {$\scriptstyle\chamC_{21}$}; 
\node (C5) at (-22.5:1.5cm) {$\scriptstyle\chamC_{2}$}; 
\end{tikzpicture}
\end{array}
&&
\begin{array}{c}
\begin{tikzpicture}[scale=1.3,bend angle=15, looseness=1,>=stealth]
\coordinate (A1) at (135:2cm);
\coordinate (A2) at (-45:2cm);
\coordinate (B1) at (153.435:2cm);
\coordinate (B2) at (-26.565:2cm);
\draw[black!30] (A1) -- (A2);
\draw[black!30] (-2,0)--(2,0);
\draw[black!30] (0,-1.8)--(0,1.8);
\node (C+) at (45:1.5cm) {$\scriptstyle\scrC_{+}$};
\node (C1) at (112.5:1.5cm) {$\scriptstyle\scrC_{1}$};
\node (C2) at (157.5:1.5cm){$\scriptstyle\scrC_{{12}}$}; 
\node (C-) at (225:1.5cm) {$\scriptstyle\scrC_{{121}}$}; 
\node (C4) at (-67.5:1.5cm) {$\scriptstyle\scrC_{{21}}$}; 
\node (C5) at (-22.5:1.5cm) {$\scriptstyle\scrC_{2}$}; 
\draw[->, bend right]  (C+) to (C1);
\draw[->, bend right]  (C1) to (C+);
\draw[->, bend right]  (C1) to (C2);
\draw[->, bend right]  (C2) to (C1);
\draw[->, bend right]  (C2) to (C-);
\draw[->, bend right]  (C-) to (C2);
\draw[<-, bend right]  (C+) to  (C5);
\draw[<-, bend right]  (C5) to  (C+);
\draw[<-, bend right]  (C5) to  (C4);
\draw[<-, bend right]  (C4) to (C5);
\draw[<-, bend right]  (C4) to  (C-);
\draw[<-, bend right]  (C-) to (C4);
\node at (78.75:0.9cm) {$\scriptstyle \Phi_1$};
\node at (78.75:1.6cm) {$\scriptstyle \Phi_1$};
\node at (135:1.075cm) {$\scriptstyle \Phi_2$};
\node at (135:1.7cm) {$\scriptstyle \Phi_2$};
\node at (198:0.9cm) {$\scriptstyle \Phi_1$};
\node at (198:1.6cm) {$\scriptstyle \Phi_1$};
\node at (258.75:0.9cm) {$\scriptstyle \Phi_2$};
\node at (258.75:1.6cm) {$\scriptstyle \Phi_2$};
\node at (315:1cm) {$\scriptstyle \Phi_1$};
\node at (315:1.75cm) {$\scriptstyle \Phi_1$};
\node at (8:0.95cm) {$\scriptstyle \Phi_2$};
\node at (8:1.6cm) {$\scriptstyle \Phi_2$};
\end{tikzpicture}
\end{array}
\end{array}
\]
The precise definition of the categories $\scrC_\chamC$ and the mutation functors depends on the setting (1)--(3).  In this example, the existence of the homomorphism \eqref{eqn: Deligne rep} is simply the statement that there is a functorial isomorphism $\Phi_1\Phi_2\Phi_1\cong\Phi_2\Phi_1\Phi_2$ whenever that makes sense.  
\end{exa}

In general, higher length braid relations hold between the mutation functors \cite{DW3}.

\begin{rem}\label{rem: calibration}
The following calibration is important.  In the setting of the minimal resolution (1), $\Phi_i$ is functorially isomorphic to the spherical twist of \cite{ST} around the sheaf $\scrO_{\Curve_i}(-1)$. In the setting (3), $\Phi_i$ is functorially isomorphic to the inverse of the Bridgeland--Chen flop functor \cite{B02, Chen}, and further $\Phi_i\Phi_i$ is the twist functor around the NC deformations of $\scrO_{\Curve_i}(-1)$ \cite{DW1, DW3}.  When $\scrO_{\Curve_i}(-1)$ is spherical, namely $\Curve_i$ is a $(-1,-1)$-curve, then $\Phi_i\Phi_i$ is the standard spherical twist.  In particular, in the $3$-fold setting the functor $\Phi_i$ should be thought of as the square root of twist functors.
\end{rem}

\begin{rem}
In the case of the minimal resolution $\scrZ\to\mathbb{C}^2/\Gamma$, the categories in each chamber are equal, not just equivalent.  As such, they can be identified, using the Weyl group $W$, and so \eqref{eqn: Deligne ver} can be improved to a homomorphism
\[
\uppi_1\big( (\mathbb{C}^n\backslash \scrH_\mathbb{C})/W \big)\to\Auteq\scrC_\chamC.
\]
This is precisely the statement that the braid group acts on any fixed $\scrC_\chamC$, generated by the mutation functors $\Phi_i$.  In the other settings (2) and (3), usually the categories in different chambers are not equal, and so we cannot make this identification.  As such, usually in the setting (2) and always in the setting (3), the homomorphism \eqref{eqn: Deligne ver} is the best possible.
\end{rem}

The main technical point below is that in all settings (1)--(3), the mutation functor $\Phi_i\cong\RHom(\mathrm{T}_i,-)$ where $\mathrm{T}_i$ is a tilting module of projective dimension one.  This fact is what allows all cases to be treated uniformly.

\subsection{Standard Hearts and Deformation Algebras}\label{subsec: cont alg}
The categories $\scrC_\chamC$ inherit a standard t-structure. Maintaining the notation in \S\ref{sec: bij MMP}, consider the following.
\begin{enumerate}
\item For the minimal resolution $\scrZ\to\mathbb{C}^2/\Gamma$, the standard t-structure on $\Db(\fmod\Uppi)$ restricts to give a bounded t-structure on $\scrC_\chamC$ with heart $\fmod(\Uppi/(e_0))$. Note that $\Uppi/(e_0)$ is the preprojective algebra of (finite) type ADE.
\item For a partial resolution $\scrY\to\mathbb{C}^2/\Gamma$, the standard t-structure on $\Db(\fmod e_\scrJ\Uppi e_\scrJ)$ restricts to give a bounded t-structure on $\scrC_\chamC$ with heart $\fmod(e_\scrJ\Uppi e_\scrJ/(e_0))$. 
\item For a flopping contraction $\scrX\to\Spec\scrR$, the standard t-structure on $\Db(\fmod \Lambda_\chamC)$ restricts to give a bounded t-structure on $\scrC_\chamC$ with heart $\fmod(\Lambda_\chamC/(e))$. 
Note that $\Lambda_{\chamC,\con}\colonequals \Lambda_\chamC/(e)\cong\underline{\End}_\scrR(M_\chamC)$ is known as the contraction algebra.
\end{enumerate}
\begin{nota}
Write $\fmod\mathrm{A}_{\chamC,\con}$ for the standard heart in $\scrC_\chamC$, where $\mathrm{A}_{\chamC,\con}$ varies depending on the setting (1)--(3), and so equals $\Uppi/(e_0)$, $e_\scrJ\Uppi e_\scrJ/(e_0)$ or $\Lambda_\chamC/(e)$.
\end{nota}
Although the precise algebra $\mathrm{A}_{\chamC,\con}$ varies, it is always finite dimensional.  Furthermore, in all cases $\mathrm{A}_{\chamC,\con}$ represents noncommutative deformations of the reduced fibre.

\subsection{Simples}\label{subsec: simples}
All the standard hearts $\fmod\mathrm{A}_{\chamC,\con}$ above have finite length, with finitely many simples.  As such, recording the action of functors on the simples becomes important.

\begin{nota}
For any chamber $\chamC$, write $\scrS_{1,\chamC}, \hdots,\scrS_{n,\chamC}$ for the simple $\mathrm{A}_{\chamC,\con}$-modules, and consider $\scrS_\chamC\colonequals \bigoplus_{i=1}^n\scrS_{i,\chamC}$.   If it is implicitly clear to which category the simples belong, we will drop the subscript $\chamC$ and write  $\scrS_{1}, \hdots,\scrS_{n}$ for the simples and $\scrS= \bigoplus\scrS_{i}$ for their sum.
\end{nota}

By \cite[3.5.8]{VdB} $\scrS_{i,\chamC}$ corresponds to the sheaf $\scrO_{\Curve_i}(-1)$ under the derived equivalence \eqref{NCCR Db}.  The labelling is consistent so that mutation $\Phi_i$ always shifts $\scrS_i$ to the right by one homological degree, in all three settings.

\begin{lem}\label{actions on some modules}
For a chamber $\chamC$, consider a length one wall crossing $s_i\colon \chamC \to \mathsf{D}$, and also the `longest element' atom $\ell\colon \chamC\to\mathsf{op}(\chamC)$.  Then the following statements hold.
\begin{enumerate}
\item\label{actions on some modules 1} $\Phi_i(\scrS_i)=\scrS_{i}[-1]$,
\item\label{actions on some modules 2}  There exists a permutation $\upsigma$ such that $\Phi_\ell(\scrS_i)\cong\scrS_{\upsigma(i)}[-1]$ for all $i=1,\hdots,n$.  
\item\label{actions on some modules 3} If $x\in\fmod \mathrm{A}_{\chamC, \con}$, then $\Phi_\ell(x)\cong y[-1]$ for some $y\in\fmod \mathrm{A}_{\mathsf{op}(\chamC),{\con}}$.
\end{enumerate}
\end{lem}
\begin{proof}
(1) and (2) are contained in \cite[6.3]{HomMMP}.  Part (3) follows easily from (2) since all such $x$ are filtered by simples.
\end{proof}

\section{Classification of Spherical Objects}\label{sec: classification}

Throughout this section, let $\scrC_\chamC$ be the category defined in \S\ref{sec: bij MMP}, and we will freely use the equivalent categories $\scrC_\chamD$ and the functors $\Phi_\upalpha$ explained in \S\ref{sec: bij MMP} and \ref{Phi alpha}.

\subsection{Bounded Regions}\label{sec: atoms and simples}

Given $x\in\scrC_\chamC$, write $\H^*(x)$ or sometimes more simply $\H^*x$ for the cohomology of $x$ with respect to the standard heart $\fmod \mathrm{A}_{\chamC,\con}$ on $\scrC_\chamC$.  So, $\H^*(x)$ always denotes cohomology with respect to the standard heart of the category to which $x$ belongs.  To illustrate this, if $x\in\scrC_\chamC$ and $\upbeta\colon\chamC\to\chamD$, then $\Phi_\upbeta(x)\in\scrC_\chamD$ and so $\H^*(\Phi_\upbeta x)$ denotes the cohomology of  $\Phi_\upbeta(x)$ with respect to the standard t-structure on $\scrC_\chamD$.

\begin{nota}\label{nota: squares 1}
Write $x\in [a,b]$, where $a\leq b$, to mean $\H^i(x)=0$ if $i<a$ and $i>b$, and write $x\in [\![a,b]\!]$ to mean $x\in [a,b]$ where furthermore $\H^a(x)\neq 0$ and $\H^b(x)\neq 0$.  
\end{nota}

There are obvious self-documenting variations, such as $[\![a,b]$.  
The above notation can be interpreted using vanishings of $\Hom$ against the direct sum of the simples $\scrS=\bigoplus\scrS_i$. The following is an easy consequence of the fact that every bounded complex admits a morphism from its (shifted) bottom cohomology, and a nonzero map to its (shifted) top cohomology, together with the fact that two objects in a heart admit no negative Exts.  A more general version of the following, and proof, can be found in \ref{bound for objects}.

\begin{lem} \label{notation via vanishings}
If $x \in \scrC$, then the following are equivalent.
\begin{enumerate}
\item $x \in [a,b]$.
\item $\Hom(\scrS, x[i]) = 0$ for all $i < a$ and $\Hom(x, \scrS[i]) = 0$ for all $i < -b$.
\end{enumerate}
Under these conditions, $x \in  [\![a,b]$ iff $\Hom(\scrS, x[a]) \neq 0$, and $x \in [a, b]\!]$ iff $\Hom(x, \scrS[-b]) \neq 0$.
\end{lem}

The following lemma is a summary of known results, and is ultimately an easy consequence of the mutation functors being induced by tilting modules of projective dimension one.  The main point is that for atoms $\upalpha$, the functor $\Phi_\upalpha$ can at worst move cohomology to the right.

\begin{lem}\label{lem:atoms}
Let $x\in\scrC_\chamC$ with $x\in[\![a,b]\!]$, and let $\upalpha\colon\chamC\to\chamD$ be an atom. Then the following statements hold.
\begin{enumerate}
\item\label{lem:atoms 0} 
$\Phi_\upalpha\cong\RHom(\mathrm{T}_\upalpha,-)$ for some $\mathrm{T}_\upalpha$ satisfying $\mathrm{pd}\, \mathrm{T}_\upalpha\leq 1$.
\item\label{lem:atoms 1} $\Phi_\upalpha(x)\in [a,b+1]$.
\item\label{lem:atoms 1B} If $\Phi_\upalpha(x)\in [a,b]$, then $\Phi_\upalpha(x)\in [a,b]\!]$.
\item\label{lem:atoms 2} If there exists $\scrS_i$ with $\Hom_\scrC(\H^bx,\scrS_i)=0$, then $\Phi_i(x)\in [a,b]$.
\item\label{lem:atoms 3} $s_i\circ\upalpha$ is an atom if and only if $\Phi_i\circ\Phi_\upalpha(\scrS)\colonequals \Phi_{s_i\circ\upalpha}(\scrS)\in [0,1]$.
\item\label{lem:atoms 4} Consider the longest element atom $\ell\colon\chamC\to\mathsf{op}(\chamC)$.  Then $\Phi_\ell(x)\in [\![ a+1, b+1]\!]$.
\end{enumerate}
\end{lem}
\begin{proof}
(1) In the flops setting, this is \cite[4.6]{HW} when $\scrX$ is smooth, and \cite[9.34]{IW9} in general. The case of partial resolutions, or the minimal resolution, is covered by \cite[5.30]{IW9}.\\
(2) Consider the spectral sequence
\[
\mathrm{E}_2^{p,q}=\H^p\bigl(\Phi_\upalpha(\H^qx)\bigr)\Rightarrow \H^{p+q}\bigl(\Phi_\upalpha x\bigr).
\]
By (1), the only non-zero $\mathrm{E}_2$ terms are $\mathrm{E}_2^{0,q}$ and $\mathrm{E}_2^{1,p}$. Thus the spectral sequence degenerates immediately, and the statement follows.\\
(3) If $\H^b(\Phi_\upalpha x)=\H^{b+1}(\Phi_\upalpha x)=0$, then in particular the whole top row of the spectral sequence in (2) must be zero, namely $\H^0(\Phi_\upalpha(\H^b x))=\H^1(\Phi_\upalpha(\H^b x))=0$. It follows that $\Phi_\upalpha(\H^b x)=0$, and so $\H^b x=0$, contradicting $x\in[\![a,b]\!]$.\\
(4) Since $y=\H^bx$ has finite length, the assumption $\Hom(y,\scrS_i)=0$ immediately implies that $\H^j(\Phi_i y)=0$ for all $j\neq 0$, by e.g.\ \cite[5.10]{HomMMP}.  Thus $\H^1(\Phi_i(\H^bx))=0$, from which the spectral sequence in (2) implies that $\H^{b+1}(\Phi_ix)=0$ and so $\Phi_i(x)\in [a,b]$.\\
(5) By \cite[6.3]{HW} in the smooth case, and \cite[10.17]{IW9} in general, $s_i\circ\upalpha$ is an atom if and only if the maximum $p$ with $\Ext^p(\scrS,\Phi_{s_i\upalpha}(\Lambda))\neq 0$ is $p=1+d$.  Since $\Lambda$ is $d$-singular CY (see e.g.\ \cite[\S2.4]{IW1}), this holds iff the maximum $p$ for which $\Ext^{d-p}(\Phi_{s_i\upalpha}(\Lambda),\scrS)\cong \H^{d-p}(\Phi_{s_i\upalpha}^{-1}\scrS)\neq 0$ is $p=1+d$.  In other words, iff the minimum $q$ for which $\H^{q}(\Phi_{s_i\upalpha}^{-1}\scrS)\neq 0$ is $q=-1$, which holds iff $\Phi_{s_i\upalpha}^{-1}(\scrS)\in[-1,0]$.  It is clear, using (2), that this is equivalent to $\Phi_{s_i\upalpha}(\scrS)\in[0,1]$.\\
(6) This is an easy consequence of \ref{actions on some modules}\eqref{actions on some modules 3}.
\end{proof}

The following technical lemma is a very mild generalisation of \cite[6.8]{SW}.  A more general statement and proof, which also avoids spectral sequences, is given in \ref{magic lemma Db 2} later.

\begin{lem} \label{magic lemma}
Let $x, y \in \scrC$ with $x \in[\![a,b]$ and $y \in [c, d]\!]$.
\begin{enumerate}
\item\label{magic lemma 1} If $\Hom_{\scrC}(y,x[a-d]) = 0$, then $\Hom_{\scrC}(\H^dy, \H^ax)=0.$
\item\label{magic lemma 2} If $x\in[\![a,b]\!]$, $a<b$, $\Hom_{\scrC}( x, x[i])=0$ for all $i<0$, then $\Hom_{\scrC}(\H^bx,\H^ax)=0$.
\end{enumerate}
\end{lem}
\begin{proof}
(1) Since $\scrC\subset \Db(\fmod\Lambda)$, by \cite[\S 4.1]{IU} there is a spectral sequence
\[
\mathrm{E}_2^{p,q}=\bigoplus_{i}\Hom_\scrC^p(\H^iy,\H^{i+q}x)\Rightarrow \Hom_\scrC^{p+q}(y,x).
\]
Since $x \in[\![a,b]$ and $y \in [c, d]\!]$, $\mathrm{E}_2^{p,q}=0$ for all $q$ with $q < a - d$.  Further, since modules have no negative Exts, $\mathrm{E}_2^{p,q}=0$ for all $p$ with $p<0$.  It follows that the bottom left term is $\mathrm{E}_2^{0, a-d} = \Hom_{\scrC}(\H^dy, \H^ax)$, which survives to give $\Hom_{\scrC}(y,x[a-d])$. Since this is zero by assumption, necessarily $\Hom_{\scrC}(\H^dy, \H^ax)=0$.\\
(2) This is an immediate consequence of (1).
\end{proof}

\subsection{Main Results}\label{sec: main results first}

The following is the first main result of this paper. It asserts that any object $x\in\scrC$ which behaves like it lies in the heart of a t-structure, namely it admits no negative self-extension groups, does lie in the heart of a t-structure.
\begin{thm}\label{make module}
Suppose that $x\in\scrC_\chamC$ satisfies $\Hom_{\scrC}( x, x[i])=0$ for all $i<0$.  Then there exists $\upbeta\colon\chamC\to\chamD$ such that the corresponding $\Phi_\upbeta\colon \scrC_\chamC\to\scrC_{\chamD}$ satisfies $\Phi_\upbeta (x)\in\fmod \mathrm{A}_{\chamD,\con}$.\end{thm}
\begin{proof}
Since $x$ is a bounded complex, there exists $a\leq b$ such that $x\in [\![a,b]\!]$.  If $a=b$ then $x$ is the shift of a module.  Consider the longest element atom $\ell\colon \chamC\to \mathsf{op}(\chamC)$, then Lemma~\ref{actions on some modules}\eqref{actions on some modules 3} shows that applying either $\Phi_\ell$ or $\Phi_\ell^{-1}$ repeatedly results in a module, over the relevant $\mathrm{A}_{\chamD,\con}$, in degree zero.

Hence we can assume that $a<b$.  The trick is to consider the set
\[
\Updelta\colonequals \{ \upalpha \mid \upalpha\mbox{ is an atom, and } \Phi_\upalpha(x)\in[a,b]\,\},
\]
which in light of \ref{lem:atoms}\eqref{lem:atoms 1} is the set of atoms which do not make $x$ homologically worse. 
Clearly the trivial atom $\chamC\to\chamC$ belongs to $\Updelta$, 
and the longest atom $\ell\colon\chamC\to\mathsf{op}(\chamC)$ does not belong to $\Updelta$ since $\Phi_\ell(x)\in[\![a+1,b+1]\!]$ by \ref{lem:atoms}\eqref{lem:atoms 4}.   

The set $\Updelta$ admits a partial order, since the set of all atoms does (see \ref{dfn: atoms order}), and further the set $\Updelta$ is finite, since the set of all atoms is finite.
Every finite poset has at least one maximal element, so consider a maximal element $\upbeta\in\Updelta$.   By definition $\Phi_\upbeta(x)\in [a,b]$, which by \ref{lem:atoms}\eqref{lem:atoms 1B} forces  $\Phi_\upbeta(x)\in [a,b]\!]$.

We claim that $\Phi_\upbeta(x)\in [a+1,b]\!]$. This implies that $\upbeta$ has improved the complex $x$, which in turn will allow us to induct.   
Since  $\Phi_\upbeta(x)\in[a,b]\!]$, the claim is simply that $\H^a(\Phi_\upbeta x)=0$. 
This will be proved by assuming that $\H^a(\Phi_\upbeta x)\neq 0$, and deriving a contradiction.

Well, if $\H^a(\Phi_\upbeta x)\neq 0$, then $\Phi_\upbeta(x)\in[\![a,b]\!]$ and so   by \ref{magic lemma}\eqref{magic lemma 2} since $a<b$
\begin{equation}
\Hom(\H^b(\Phi_\upbeta x), \H^a(\Phi_\upbeta x))=0.\label{eqn: top bottom 1}
\end{equation}
Since $\H^a(\Phi_\upbeta x)\neq 0$, pick a simple $\scrS_i\hookrightarrow \H^a(\Phi_\upbeta x)$.  
From \eqref{eqn: top bottom 1} it follows that
\[
\Hom(\H^b(\Phi_\upbeta x),\scrS_i)=0.
\]
Applying \ref{lem:atoms}\eqref{lem:atoms 2} to $\Phi_\upbeta(x)$ shows that $\Phi_i\Phi_\upbeta(x)\in [a,b]$, i.e.\ $\Phi_{s_i\upbeta}(x)\in[a,b]$.  
Now if $\upbeta=\Id$ then clearly $s_i\circ \upbeta$ is an atom, and so $\Phi_{s_i\upbeta}(x)\in[a,b]$ contradicts the maximality of $\upbeta$.  
Hence we can assume that $\upbeta\neq \Id$.  

Since the longest atom does not belong to $\Updelta$, and $\upbeta\neq\Id$, necessarily $\Phi_\upbeta(\scrS)\in[\![0,1]\!]$. 
The assumption $x \in[\![a,b]\!]$ implies that $\Hom(\Phi_\upbeta(\scrS), \Phi_\upbeta(x) [a-1]) \cong \Hom(\scrS, x[a-1]) = 0$ by \ref{notation via vanishings}, and so it follows from \ref{magic lemma} that
\begin{equation}
0=
\Hom(\H^1(\Phi_\upbeta\scrS), \H^a(\Phi_\upbeta x)).
\label{eqn: top bottom 2}
\end{equation}
The injection $\scrS_i\hookrightarrow \H^a(\Phi_\upbeta(x))$ together with \eqref{eqn: top bottom 2} implies that
\[
\Hom(\H^1(\Phi_\upbeta \scrS ),\scrS_i)=0.
\]
Applying \ref{lem:atoms}\eqref{lem:atoms 2} to $\Phi_\upbeta(\scrS)$ shows that $\Phi_i\Phi_\upbeta(\scrS)\in[0,1]$, so by  \ref{lem:atoms}\eqref{lem:atoms 3} $s_i\circ\upbeta$ is an atom.  
Again, $\Phi_{s_i\upbeta}(x)\in[a,b]$ contradicts the maximality of $\upbeta$. Thus $\H^a(\Phi_\upbeta(x))=0$, as claimed.

We conclude that $\Phi_\upbeta(x)\in[a+1,b]\!]$, and so the complex $\Phi_\upbeta(x)$ is strictly shorter than that of $x$.  
Replacing $x$ by $\Phi_\upbeta(x)$ we can then repeat the above argument, 
and by doing this find a composition $\upgamma$ such that $\Phi_\upgamma(x)$ lies in precisely one homological degree.  
As in the first paragraph, repeatedly applying either $\Phi_\ell$ or $\Phi_\ell^{-1}$ results in a module, over the relevant $\mathrm{A}_{\chamD,\con}$, in degree zero.
\end{proof}

From the viewpoint of our applications in \S\ref{sec4: geo cors}, the following consequence of \ref{make module} is important.  The power in the statement, and the reason that the proof is a little involved, is precisely since part (1) considers direct sums.  We explain  in \ref{rem: how to improve} that for the minimal resolution, and other special cases, an easier proof of (2) exists.

\begin{cor}\label{cor 2 3folds}
Suppose that $x\in\scrC_\chamC$ satisfies $\Hom( x, x[i])=0$ for all $i<0$.
\begin{enumerate}
\item\label{cor 2 3folds 1} If $\dim_{\mathbb{C}}\Hom(x_i, x_j)\cong \mathbb{C}\updelta_{ij}$ for all indecomposable summands $x_i, x_j$ of $x$, then there exists a subset $\scrI$ and a path $\upgamma\colon \chamC\to\mathsf{E}$ such that $\Phi_\upgamma(x)\cong \bigoplus_{i\in \scrI}\scrS_{i}$.
\item\label{cor 2 3folds 2} If $\dim_{\mathbb{C}}\Hom(x, x)\cong \mathbb{C}$, then there exists $\upgamma\colon \chamC\to\mathsf{E}$ and $i$ such that $\Phi_\upgamma(x)\cong \scrS_i$.
\end{enumerate}
\end{cor}
\begin{proof}
We prove (1), since (2) is a special case.  Consider first the $3$-fold setting. 

By \ref{make module}, there exists $\upbeta\colon\chamC\to\chamD$ such that $\Phi_\upbeta (x)\cong y$, where $y$ is a $\Lambda_{\chamD,\con}$-module in homological degree zero.  Now by assumption $x$ is a direct sum of Hom-orthogonal indecomposable summands, hence so too is $y$.  Thus in the language of \S\ref{sec: silting prelims}, $y$ is a semibrick in the category $\fmod\Lambda_{\chamD,\con}$.  Since $\Lambda_{\chamD,\con}$ is silting discrete \cite[3.12]{August1}, by \cite{Asai} $y$ is a subset of a 2-simple minded collection, which by definition and the Koenig--Yang bijections implies that $y$ is a subset of simples in the heart $\scrH$ of an intermediate bounded t-structure in $\Db(\fmod\Lambda_{\chamD,\con})$.

Now by \cite[1.1]{August}, for every atom $\upalpha\colon\chamD\to\chamE$ there exists an autoequivalence $F_\upalpha$ such that the following diagram commutes
\begin{equation}
\begin{array}{c}
\begin{tikzpicture}
\node (A) at (0,0) {$\Db(\fmod\Lambda_{\chamD,\con})$};
\node (B) at (3.5,0) {$\Db(\fmod\Lambda_\chamD)$};
\node (a) at (0,-1.5) {$\Db(\fmod\Lambda_{\chamE,\con})$};
\node (b) at (3.5,-1.5) {$\Db(\fmod\Lambda_\chamE)$};
\draw[->] (A)--node[above]{$\scriptstyle \rm res$}(B);
\draw[->] (A)--node[left]{$\scriptstyle F_\upalpha$}(a);
\draw[->] (B)--node[right]{$\scriptstyle \Phi_\upalpha$}(b);
\draw[->] (a)--node[above]{$\scriptstyle \rm res$}(b);
\end{tikzpicture}
\end{array}\label{Jenny comm diag}
\end{equation}
and further 
by \cite{August} the heart of every intermediate bounded t-structure in $\Db(\fmod\Lambda_{\chamD,\con})$ equals $F_{\alpha}^{-1}(\fmod\Lambda_{\chamE,\con})$ for some atom $\upalpha\colon \chamD\to\chamE$.  

Thus there exists a subset $\scrI$  such that $F_\upalpha^{-1}(\bigoplus_{i\in \scrI}\scrS_{i,\chamE})=y$. Since the diagram \eqref{Jenny comm diag} commutes, necessarily $\Phi_\upalpha(y)\cong\bigoplus_{i\in \scrI}\scrS_{i,\chamE}$, and so the path $\upalpha\circ\upbeta\colon \chamC\to\chamE$ satisfies $\Phi_{\upalpha\circ\upbeta}(x)=\Phi_{\upalpha}\Phi_\upbeta(x)=\Phi_\upalpha(y)=\bigoplus_{i\in \scrI}\scrS_{i,\chamE}$, as required. This completes the proof in the $3$-fold setting.

\medskip
Consider now the surfaces setting.  By \ref{make module}, there exists $\upbeta\colon\chamC\to\chamD$ such that $\Phi_\upbeta (x)\cong y$, where $y$ is a $(e_\scrI\Uppi e_{\scrI})_{\con}$-module in homological degree zero.  Now $\scrC_\chamD\subset\Db(\fmod e_\scrI\Uppi e_\scrI)$ say, where $\Uppi$ is the preprojective algebra of affine ADE, and $e_\scrI$ is some idempotent.  As in \cite[\S3]{Toda} or \cite[\S4.3]{KM}, we can construct a $3$-fold flopping contraction $\scrX\to\Spec\scrR$, where $\scrX$ has Gorenstein terminal singularities (where $\scrX$ tilts to $\Lambda_\chamD$), such that the generic hyperplane section $g\in\scrR$ satisfies $(\Lambda_\chamD)/g\cong e_\scrI\Uppi e_\scrI$. 

Furthermore, for any atom $\upalpha\colon \chamD\to\chamE$, then \cite{IW9} shows that the following diagram commutes. 
\[
\begin{tikzpicture}
\node (A) at (0,0) {$\Db(\fmod e_{\scrI}\Uppi e_{\scrI})$};
\node (B) at (3.5,0) {$\Db(\fmod\Lambda_\chamD)$};
\node (a) at (0,-1.5) {$\Db(\fmod e_{\scrK}\Uppi e_{\scrK})$};
\node (b) at (3.5,-1.5) {$\Db(\fmod\Lambda_\chamE)$};
\draw[->] (A)--node[above]{$\scriptstyle \rm res$}(B);
\draw[->] (A)--node[left]{$\scriptstyle\Phi_\upalpha$}(a);
\draw[->] (B)--node[right]{$\scriptstyle\Phi_\upalpha$}(b);
\draw[->] (a)--node[above]{$\scriptstyle \rm res$}(b);
\end{tikzpicture}
\]
Since $y$ is a module in homological degree zero, $\mathrm{res}(y)\in\scrC_\chamD$ is likewise still a module, and thus it has no negative Exts.  Further by adjunction 
\[
\Hom_{\scrC}(\mathrm{res}(y_i),\mathrm{res}(y_j))\cong\mathbb{C}\updelta_{ij} 
\] 
for all indecomposable summands $y_i, y_j$, and thus $\mathrm{res}(y)$ is a semibrick.  The $3$-fold proof above shows that there exists some atom $\upalpha$ such that $\Phi_\upalpha(\mathrm{res}(y))\cong \scrU$, where $\scrU=\bigoplus_{i\in \scrI}\scrS_i$ for some subset $\scrI$.  The commutativity of the diagram implies that $\mathrm{res}(\Phi_\upalpha(y))\cong \scrU$. But also $\mathrm{res}(\scrU)=\scrU$, so since $\mathrm{res}$ reflects isomorphisms when restricted to the module category, $\Phi_\upalpha(y)\cong  \scrU$ follows.
\end{proof}

\subsection{Algebraic Corollaries}\label{subsec: alg cors}
 It is a result of Crawley--Boevey \cite[Lemma 1]{BCB} that every brick for the preprojective algebra $\Uppi$ of type ADE has dimension vector equal to a root of the corresponding Dynkin diagram.  This subsection generalises this to contracted preprojective algebras $\Upgamma=e_\scrI\Uppi e_\scrI$, and also to contraction algebras $\Lambda_{\con}$. 

Recall the notion of restricted roots from \S\ref{subsec:arrangements}, and the conventions in \ref{rem: subset conventions} and \eqref{eqn: idempotent convention}.  The following is then a consequence of \ref{cor 2 3folds}.
\begin{cor}\label{ePie dim vectors}
Consider the projective algebra $\Uppi$ of type \textnormal{ADE}, and $\scrI\subset\Updelta$.  
Then every brick in $\fmod e_\scrI\Uppi e_\scrI$ has dimension vector equal to a primitive restricted root. 
\end{cor}
\begin{proof}
If $x\in\fmod\Lambda_{\con}$ is a brick, then by \ref{cor 2 3folds} there exists $\upbeta$ such that $\Phi_\upbeta(\scrS_i)\cong x$.  Writing $\upvarphi_\upbeta$ for the action of $\Phi_\upbeta$ on $\mathrm{K}_0(\scrC)$, it follows that $\upvarphi_\upbeta(\mathbf{e}_i)=\underline{\mathrm{dim}}\,x$.  But, as explained in e.g.\ \cite[\S5.1, Rem 5.12]{NW}, $\upvarphi_\upbeta$ can be written as a  composition of matrices $\mathsf{N}_i$, each of which takes primitive restricted roots to primitive restricted roots.  Since $\mathbf{e}_i$ is clearly a primitive restricted root, so too is $\underline{\mathrm{dim}}\,x$.
\end{proof}

\begin{rem}
It should also be possible to prove \ref{ePie dim vectors} using the fact that $c$-vectors are the dimension vectors of bricks, and these are dual to $g$-vectors \cite{Hipolito}.  The $g$-vectors for some $e_\scrI\Uppi e_\scrI$ are known \cite[\S6]{IW9}, but these results currently require some restrictions on $\scrI$. 
\end{rem}

Recall that to every contraction algebra $\Lambda_{\con}$, there exists a central element $g$ in the radical of $\Lambda_{\con}$ such that $\Lambda_{\con}/(g)\cong e_\scrI\Uppi e_\scrI$ for some $\scrI\subset\Updelta$, where $\Updelta$ is ADE \cite{DW1,DW3}.
\begin{cor}\label{cor contraction brick dim}
If $\Lambda_{\con}$ is the contraction algebra associated to a $3$-fold flopping contraction $\scrX\to\Spec\scrR$ where $\scrX$ is at worst Gorenstein terminal, then every brick in $\fmod\Lambda_{\con}$ has dimension vector equal to a primitive restricted root.  
\end{cor}
\begin{proof}
This is an immediate consequence of \ref{ePie dim vectors}, since by the standard reduction theorems \cite[\S4]{EJR} and \cite[1.29]{Asai}, bricks on $\Lambda_{\con}$ are the same as bricks for $\Lambda_{\con}/(g)$.
\end{proof}

\begin{rem}\label{rem: how to improve}
The proof of \ref{cor 2 3folds}\eqref{cor 2 3folds 2}, but alas not \ref{cor 2 3folds}\eqref{cor 2 3folds 1}, can be simplified in the case of the minimal resolution, or in the case of $3$-fold crepant resolutions which slice to the minimal resolution.  Indeed, in the case of the minimal resolution Crawley--Boevey \cite[Lemma 1]{BCB} implies that every brick $x$ has dimension vector equal to a real root.  Hence there exists a sequence of Weyl reflections such that $s_{i_t}\cdot\hdots\cdot s_{i_1}\cdot\,\underline{\mathrm{dim}}(x)=\mathbf{e}_i$ for some $i$.  

Now, since $x$ is a brick, either $\Hom(x,\scrS_{i_1})=0$ or $\Hom(\scrS_{i_1},x)=0$.  Set $F_1\colonequals\Phi_{i_1}$ in the first case, or $F_1\colonequals\Phi_{i_1}^{-1}$ in the second, then $F_1(x)$ is a module, of dimension vector $s_{i_1}\cdot\,\underline{\mathrm{dim}}(x)$.  Clearly $F_1(x)$ is still a brick, so we can repeat the argument. At each stage, we can choose $F_j$ equal to either $\Phi_{i_j}$ or $\Phi_{i_j}^{-1}$ to guarantee that $F_j\hdots F_1(x)$ remains a module, and by doing this $F_t\hdots F_1(x)$ is a module of dimension vector equal to $\mathbf{e}_i$, a simple root.  Thus $F_t\hdots F_1(x)\cong\scrS_i$. If the analogue of Crawley--Boevey's result \ref{ePie dim vectors} can be proved first, then the above argument simplifies the proof of \ref{cor 2 3folds}\eqref{cor 2 3folds 2} in all cases.
\end{rem}

\section{Classification of \textnormal{t}-Structures}\label{sec: t-st classification}

The previous section proved that for every object $x\in\scrC$ with no negative Exts, there exists $\Phi_\upbeta$ such that $\Phi_\upbeta(x)$ is concentrated in homological degree zero.  In particular,  $x\in\scrC$ belongs to a standard heart in the groupoid.   A second and related problem is whether, given the heart $\scrH$ of an arbitrary bounded t-structure, there exists $\Phi_\upgamma$ such that $\Phi_\upgamma(\scrH)$ is a standard heart.   Whilst the problems are similar,  one does not obviously imply the other.  A priori, the heart $\scrH$ does not have finite length, and so the element-wise arguments of the previous section do not apply.  

This section strengthens the results of the previous section to also solve the second problem, and thus fully classifies all bounded t\nobreakdash-structures on $\scrC$ in terms of the orbits under the groupoid action.  This is new in all cases, even for the minimal resolution $\scrZ\to\mathbb{C}^2/\Gamma$.

\subsection{t-Structure Generalities}
Let $\scrT$ be a triangulated category with full additive subcategory $\scrH$.  Then, as explained in e.g.\ \cite[2.3]{Bridgeland_local}, $\scrH$ is the heart of a bounded t-structure on $\scrT$ if and only if 
\begin{enumerate}
\item\label{t1} $\Hom_\scrT(x,y[j])=0$ for all $x,y\in\scrH$ and all $j<0$.
\item\label{t2} For every non-zero object $x\in\scrT$ there exists integers $a\leq b$ and triangles  
\[
\begin{tikzpicture}
\node (A) at (2,0) {$0=x_{a-1}$};
\node (B) at (4,0) {$x_a$};
\node (C) at (6,0) {$x_{a+1}$};
\node (E) at (8,0) {$\hdots$};
\node (F) at (10,0) {$x_{b-1}$};
\node (G) at (12,0) {$x_{b}=x$};
\node (a) at (3,-1) {$h_{a}$};
\node (b) at (5,-1) {$h_{a+1}$};
\node (c) at (7,-1) {$h_{a+2}$};
\node (f) at (11,-1) {$h_{b}$};
\draw[->] (A)--(B);
\draw[->] (B)--(C);
\draw[->] (C)--(E);
\draw[->] (E)--(F);
\draw[->] (F)--(G);
\draw[->](B)--(a);
\draw[->](C)--(b);
\draw[->](G)--(f);
\draw[densely dotted,->] (a)--(A);
\draw[densely dotted,->] (b)--(B);
\draw[densely dotted,->] (c)--(C);
\draw[densely dotted,->] (f)--(F);
\end{tikzpicture}
\]
such that $h_i[i]\in\scrH$ for all $i$.
\end{enumerate}
As is standard, the objects $h_i[i]\in\scrH$ are called the cohomology objects of $x$ with respect to $\scrH$, and will be written $\H_\scrH^i(x)\colonequals h_i[i]$. The following mirrors \ref{nota: squares 1}.

\begin{nota}\label{not: 5.1}
Let $\scrH$ be the heart of a bounded t-structure in a triangulated category $\scrT$.  For integers $c\leq d$, write $x\in [c,d]_{\scrH}$ to mean $\H_\scrH^i(x)=0$ for all $i<c$ and for all $i>d$.  Write $x\in [\![c,d]_{\scrH}$ if $x\in [c,d]_{\scrH}$ and furthermore $\H^c_\scrH(x)\neq 0$.
\end{nota}
As in \S\ref{sec: atoms and simples} we will use other self-documenting variations, such as $[c,d[\!]_{\scrH}$ and $[\![c,d[\!]_{\scrH}$. The following generalises \ref{magic lemma}, and satisfyingly its proof avoids the use of spectral sequences.

\begin{lem} \label{magic lemma Db 2}
Let $\scrH$ be the heart of a bounded t-structure in a triangulated category $\scrT$. If $x,y\in\scrT$ with $x\in [\![ a,b]\!]_{\scrH}$ and $y\in [\![ c,d]\!]_{\scrH}$, then
\[
\Hom_{\scrT}(y,x[a-d])\cong \Hom_{\scrT}(\H_\scrH^d(y),\H_\scrH^a(x))
\]
\end{lem}
\begin{proof}
We prove the case when $a<b$ and $c<d$, since the other cases are degenerate. To ease notation set $\H^i=\H^i_\scrH$, and consider the sequence of triangles \eqref{t2} for $x$, and the analogous sequence
\begin{equation}
\begin{array}{c}
\begin{tikzpicture}
\node (A) at (2,0) {$0=y_{c-1}$};
\node (B) at (4,0) {$y_c$};
\node (C) at (6,0) {$y_{c+1}$};
\node (E) at (8,0) {$\hdots$};
\node (F) at (10,0) {$y_{d-1}$};
\node (G) at (12,0) {$y_{d}=y$};
\node (a) at (3,-1) {$h'_{c}$};
\node (b) at (5,-1) {$h'_{c+1}$};
\node (c) at (7,-1) {$h'_{c+2}$};
\node (f) at (11,-1) {$h'_{d}$};
\draw[->] (A)--(B);
\draw[->] (B)--(C);
\draw[->] (C)--(E);
\draw[->] (E)--(F);
\draw[->] (F)--(G);
\draw[->](B)--(a);
\draw[->](C)--(b);
\draw[->](G)--(f);
\draw[densely dotted,->] (a)--(A);
\draw[densely dotted,->] (b)--(B);
\draw[densely dotted,->] (c)--(C);
\draw[densely dotted,->] (f)--(F);
\end{tikzpicture}
\end{array}
\label{triangles for y}
\end{equation}
for $y$. Using the vanishing of negative Ext groups \eqref{t1}, applying $\Hom_\scrT(-,\scrH)$ it is easy to verify that
\begin{equation}
\Hom_\scrT(x_t,\scrH[-t-j])=0=\Hom_\scrT(y_t,\scrH[-t-j])\mbox{ for all }j\geq 1.\label{induction vanishing}
\end{equation}
Now applying $\Hom_\scrT(y[d],-)$ to the sequence of triangles \eqref{t2} for $x$, and temporarily dropping $\Hom$ from the notation, gives the following exact sequences.
\[
\begin{tikzpicture}
\node (a1) at (0,0) {\small${}_\scrT(y[d],h_b[a-1])$};
\node (a2) at (3.5,0) {\small${}_\scrT(y[d],x_{b-1}[a])$};
\node (a3) at (6.5,0) {\small${}_\scrT(y[d],x[a])$};
\node (a4) at (10,0) {\small${}_\scrT(y[d],h_b[a])$};
\node (b1) at (0,-0.5) {\small${}_\scrT(y[d],h_{b-1}[a-1])$};
\node (b2) at (3.5,-0.5) {\small${}_\scrT(y[d],x_{b-2}[a])$};
\node (b3) at (6.5,-0.5) {\small${}_\scrT(y[d],x_{b-1}[a])$};
\node (b4) at (10,-0.5) {\small${}_\scrT(y[d],h_{b-1}[a])$};
\node at (6.5,-0.9) {$\vdots$};
\node (c1) at (0,-1.5) {\small${}_\scrT(y[d],h_{a+2}[a-1])$};
\node (c2) at (3.5,-1.5) {\small${}_\scrT(y[d],x_{a+1}[a])$};
\node (c3) at (6.5,-1.5) {\small${}_\scrT(y[d],x_{a+2}[a])$};
\node (c4) at (10,-1.5) {\small${}_\scrT(y[d],h_{a+2}[a])$};
\node (d1) at (0,-2) {\small${}_\scrT(y[d],h_{a+1}[a-1])$};
\node (d2) at (3.5,-2) {\small${}_\scrT(y[d],x_{a}[a])$};
\node (d3) at (6.5,-2) {\small${}_\scrT(y[d],x_{a+1}[a])$};
\node (d4) at (10,-2) {\small${}_\scrT(y[d],h_{a+1}[a])$};

\draw[->] (a1) -- (a2);
\draw[->] (a2) -- (a3);
\draw[->] (a3) -- (a4);
\draw[->] (b1) -- (b2);
\draw[->] (b2) -- (b3);
\draw[->] (b3) -- (b4);
\draw[->] (c1) -- (c2);
\draw[->] (c2) -- (c3);
\draw[->] (c3) -- (c4);
\draw[->] (d1) -- (d2);
\draw[->] (d2) -- (d3);
\draw[->] (d3) -- (d4);
 \draw [dotted] (-1.5,.5) -- (-1.5,-2.5) -- (1.5,-2.5) -- (1.5,.5) -- cycle;
  \draw [dotted] (8.5,.5) -- (8.5,-2.5) -- (11.5,-2.5) -- (11.5,.5) -- cycle;
\end{tikzpicture}
\]
Since $y=y_d$ it follows from \eqref{induction vanishing} that all Hom sets within the dotted boxes are zero. Hence, since further $x_a\cong h_a$, there are thus isomorphisms
\[
\Hom_\scrT(y[d],x[a])\cong\hdots\cong\Hom_\scrT(y[d],x_{a}[a])\cong\Hom_\scrT(y[d],\H^a(x)).
\]
Now applying $\Hom_\scrT(-,\H^a(x))$ to the rightmost triangle in \eqref{triangles for y} gives an exact sequence
\[
{}_\scrT(y_{d-1}[d+1],\H^a(x))
\to
{}_\scrT(\H^d(y),\H^a(x))
\to
{}_\scrT(y[d],\H^a(x))
\to
{}_\scrT(y_{d-1}[d],\H^a(x))
\]
where the outer two sets are again zero by \eqref{induction vanishing}. It follows that
\[
\Hom_\scrT(y[d],x[a])\cong\Hom_\scrT(y[d],\H^a(x))\cong\Hom_\scrT(\H^d(y),\H^a(x)).\qedhere
\]  
\end{proof}

It is also possible to detect membership of $[a,b]_\scrA$ using Hom vanishings. The following result is the more general version of  \ref{notation via vanishings}.

\begin{lem}\label{bound for objects} 
Let $\scrA$ be the heart of a bounded t-structure, $x \in \scrT$ an object, and let $a \leq b$ be integers.  Then the following are equivalent.
\begin{enumerate}
\item $x \in [a,b]_{\scrA}$.
\item $\Hom(\scrA, x[i]) = 0$ for all $i < a$ and $\Hom(x, \scrA[i]) = 0$ for all $i < -b$.
\end{enumerate}
Furthermore $x \in  [\![a,b]$ iff $\Hom(\scrA, x[a]) \neq 0$, and $x \in [a, b]\!]$ iff $\Hom(x, \scrA[-b]) \neq 0$
\end{lem}
\begin{proof}
If $x \in [\![c,d]\!]$, then there is an exact triangle
\[
\uptau_{\leq 0}(x[c])\to x[c] \to \uptau_{>0}(x[c]) \to 
\]
Now $\H^c(x) = (\uptau_{\leq 0} \uptau_{\geq 0})(x[c])$, so since $x[c]=\uptau_{\geq 0}(x[-c])$ by the assumption $x \in [\![c,d]\!]$, the above triangle gives a non-zero morphism $\H^c(x) \to x[c]$.  In a similar way, there exists non-zero map $x\to\H^d_\scrA(x)[-d]$.\\
(2)$\Rightarrow$(1) The above paragraph shows that if $\Hom(\scrA, x[i]) = 0$ for $i < a$, then the lowest nonzero cohomology of $x$ must be $\geq a$.  Similarly, if $\Hom(x, \scrA[i]) = 0$ for $i < -b$ hold, then the highest nonzero cohomology of $x$ must be $\leq b$.

\noindent
(1)$\Rightarrow$(2)
Conversely, let $x \in \scrT$ and assume that $x\in [\![c,d]\!]_{\scrA}$ for some $a \leq c \leq d \leq b$ .
To prove the vanishings, we use induction on $d-c$.
If $d = c$, then $x$ is an object in $\scrA$ up to shift, and the vanishings are trivial.
When $d-c > 0$, consider the triangle
\[ 
\H_{\scrA}^c(x)[-c] \to x \to x' \to \H_{\scrA}^c(x)[-c+1] .
\]
with $x'\in[c+1, d]\!]_{\scrA}$. Applying the inductive hypothesis to $x'$ implies that for all $i < a (\leq c)$
\[ 
\Hom(\scrA, x[i]) \cong \Hom(\scrA, \H_{\scrA}^c(x)[-c+i]) = 0 
\]
which proves one of the desired vanishings.
The other vanishing is similar.

The final statement follows from the first paragraph.
\end{proof}

By a slight abuse of notation, if $\scrA$ and $\scrH$ are the hearts of two bounded t-structures, write $\scrH\in[a,b]_\scrA$ if $x\in[a,b]_\scrA$ for all $x\in\scrH$.

\begin{cor} \label{bound for hearts}
Let $\scrA$ and $\scrH$ be the hearts of two bounded t-structures, and let $a \leq b$ be integers.  Then the following are equivalent.
\begin{enumerate}
\item $\scrH \in [a,b]_{\scrA}$.
\item $\Hom(\scrA, \scrH[i]) = 0$ for all $i < a$ and $\Hom(\scrH, \scrA[i]) = 0$ for all $i < -b$.
\end{enumerate}
In particular, $\scrH \in [a,b]_{\scrA}$ iff $\scrA \in [-b,-a]_{\scrH}$
\end{cor}
\begin{proof}
(1)$\Leftrightarrow$(2) follows by applying \ref{bound for objects} to all objects $x$ of $\scrH$.  The final statement follows from the symmetry of $\scrA$ and $\scrH$ in the condition (2).
\end{proof}

\begin{cor} \label{bound of hearts by smc}
Let $\scrA$ and $\scrH$ be the hearts of two bounded t-structures, and suppose that $\scrA$ is a length heart with finitely many simples, where $\scrS$ is their direct sum. Then there exist integers $a \leq b$ such that $\scrH \in [a,b]_{\scrA}$.
\end{cor}

\begin{proof}
The boundedness of the t-structure giving $\scrH$ implies that the cohomology of the single object $\scrS$ with respect to $\scrH$ is bounded, 
and hence $\scrS \in [c,d]_{\scrH}$ for some $c \leq d$.
Since $\scrA$ is the extension closure of $\scrS$ in $\scrT$, the above bound on $\scrS$ implies that $\scrA \in [c,d]_{\scrH}$. By \ref{bound for hearts}, this is equivalent to $\scrH \in [-d,-c]_{\scrA}$.
\end{proof}

\subsection{Main Results}
The above results allow us to easily extend \ref{make module}.  The proof of the following reverts to the conventions of \S\ref{sec: atoms and simples}, namely that  unadorned $[a,b]$ are by default with respect to cohomology of the standard t-structure on $\scrC_\chamC$.

\begin{thm}\label{cor: no exotic}
If $\scrH$ is the heart of a bounded t-structure on $\scrC_\chamC$, then there exists $\upbeta\colon \chamC\to\chamD$ such that $\Phi_\upbeta(\scrH)$ is the standard heart on $\scrC_\chamD$.    In particular, $\scrH$ is a finite length category, with finitely many simples.
\end{thm}
\begin{proof}
Since the standard heart on $\scrC_\chamC$ is finite length with finitely many simples, by \ref{bound of hearts by smc}, there exists integers $a \leq b$ such that $\scrH \in [\![a,b]\!]$.
If $a = b$, then $\scrH$ is a contained in a shifted standard heart, and thus it equals that shifted standard heart since if the heart of a bounded t-structure includes inside another, they must be equal.  The result then follows from \ref{lem:atoms}\eqref{lem:atoms 4}, since applying either $\Phi_\ell$ or $\Phi_\ell^{-1}$ repeatedly takes this to the standard heart in degree zero.

Now assume $a < b$ and consider the set
\[
\Updelta\colonequals \{ \upalpha \mid \upalpha\mbox{ is an atom, and } \Phi_\upalpha(\scrH)\in[a,b]\,\}.
\]
Since $\Updelta$ contains the trivial atom, this is a non-empty partially ordered finite set, and hence there exists a maximal element $\upbeta \in \Updelta$.
For this $\upbeta$ we claim that $\Phi_{\upbeta}(\scrH) \in [a+1, b]$. We assume the contrary, namely there exists an object $x \in \scrH$ such that $\Phi_{\upbeta}(x) \in [\![a,b]$, and will derive a contradiction.  Since $\H^a(\Phi_{\upbeta}x)\neq 0$, fix a simple $\scrS_i \hookrightarrow \H^a(\Phi_{\upbeta}x)$.

Let $y \in \scrH$ be any object, then we claim that $\Phi_i\Phi_{\upbeta}(y) \in [a,b]$. If $\Phi_{\upbeta}(y) \in [a, b-1]$, then this follows from \ref{lem:atoms}\eqref{lem:atoms 1}, hence assume that $\Phi_{\upbeta}(y) \in [a, b]\!]$, so $\H^b(\Phi_{\upbeta}y) \neq 0$.
Since $x$ and $y$ are objects of the heart $\scrH$, the vanishing $\Hom(\Phi_{\upbeta}y, \Phi_{\upbeta}x[j]) = 0$ holds for all $j < 0$.
Thus it follows from \ref{magic lemma} (or \ref{magic lemma Db 2}) that
\begin{equation}
0=
\Hom(\H^b(\Phi_\upbeta y), \H^a(\Phi_\upbeta x)).
\label{eqn: top bottom 2 heart}
\end{equation}
The injection $\scrS_i\hookrightarrow \H^a(\Phi_\upbeta(x))$ together with \eqref{eqn: top bottom 2 heart} implies that
\[
\Hom(\H^b(\Phi_\upbeta y ),\scrS_i)=0.
\]
Applying \ref{lem:atoms}\eqref{lem:atoms 3} to $\Phi_{\upbeta}(y)$ gives $\Phi_i\Phi_{\upbeta}(y) \in [a,b]$.  Since this holds for all $y\in\scrH$, it follows that $\Phi_i\Phi_{\upbeta}(\scrH) \in [a,b]$.

Since $\Phi_{\upbeta}(x) \in [\![a, b]$, exactly the same argument as in the proof of \ref{make module} shows that $s_i \circ \upbeta$ is an atom, and thus $s_i \circ \upbeta$ remains an element of $\Updelta$.  This is the desired contradiction, as $\upbeta$ is maximal in $\Updelta$, and thus $\Phi_{\upbeta}(\scrH) \in [a+1, b]$ holds, as claimed.

Repeating the above argument, there exists $\upbeta\colon\chamC\to\chamD$ such that $\Phi_{\upalpha}(\scrH)$ is contained in a shifted standard heart. As in the first paragraph, applying either $\Phi_\ell$ or $\Phi_\ell^{-1}$ repeatedly takes this to degree zero.
\end{proof}

\begin{rem}
It is also possible to prove \ref{cor: no exotic} by adapting the proof of \cite[Lemma 9]{David}.  The main benefit of the cohomology argument here is that it is uniform: the same technique that works for a single brick also works for the heart of a bounded t-structure.

The key fact required to adapt the proof of \cite[Lemma 9]{David} to the geometric setting $\scrC$ is that \emph{all} standard hearts from the groupoid (and thus all the hearts that appear in the \cite{David} induction) admit only finitely many torsion theories.   In the $3$-fold setting all contraction algebras are silting discrete, so the fact there are only finitely many torsion theories is the last statement in \cite[3.8]{DIJ}.  In the surfaces setting(s), by Asai \cite[1.29]{Asai} bricks modules on $e_\scrI\Uppi e_\scrI$ are the same as on $\Lambda_{\con}$, and so in particular there only finitely many.  Thus by \cite[1.4]{DIJ} each $e_\scrI\Uppi e_\scrI$ is $\uptau$-tilting finite, and so by \cite[3.8]{DIJ} each $\fmod e_\scrI\Uppi e_\scrI$ has only finitely many torsion theories.  See also \cite{AMY}.
\end{rem}

\section{Geometric Corollaries}\label{sec4: geo cors}

This section translates the algebraic results of \S\ref{sec: classification} and \S\ref{sec: t-st classification} into geometric language, then uses these results to create new geometric corollaries.

\subsection{Surfaces}
For the minimal resolution $\scrZ\to\mathbb{C}^2/\Gamma$,  the braid group $\Br$ acts on $\scrC$  defined in \eqref{eqn: define C intro}, generated by spherical twists in the objects $\scrO_{\Curve_1}(-1),\hdots,\scrO_{\Curve_n}(-1)$.

\begin{cor}\label{thm: minimal main}
Consider $\scrZ\to\mathbb{C}^2/\Gamma$, and let $x\in\scrC$.  Then the following hold. 
\begin{enumerate}
\item\label{thm: minimal main 1} If $\Hom_\scrC(x,x[j])=0$ for all $j<0$, then there exists $T\in\Br$ such that $T(x)$ is a concentrated in homological degree zero.
\item\label{thm: minimal main 2} Every spherical object in $\scrC$ belongs to the orbit, under the action of the braid group, of the objects $\scrO_{\Curve_1}(-1),\hdots,\scrO_{\Curve_n}(-1)$.
\end{enumerate}
Furthermore, the heart of every bounded t-structure on $\scrC$ is the image, under the action of the group $\Br$, of the module category of the preprojective algebra of (finite) ADE type.
\end{cor}
\begin{proof}
Under the derived equivalence $\Psi_\scrZ$ in \eqref{NCCR Db}, the geometric category $\scrC$ defined in \eqref{eqn: define C intro} corresponds to the $\scrC_\chamC$ in \S\ref{sec: bij MMP}.  Furthermore, the mutation functors are functorially isomorphic to spherical twist, via the following commutative diagram.
\[
\begin{tikzpicture}
\node (A) at (0,0) {$\Db(\coh\scrZ)$};
\node (B) at (3,0) {$\Db(\fmod\Uppi)$};
\node (a) at (0,-1.5) {$\Db(\coh\scrZ)$};
\node (b) at (3,-1.5) {$\Db(\fmod\Uppi)$};
\draw[->] (A)--node[above]{$\scriptstyle \Psi_\scrZ$}(B);
\draw[->] (a)--node[above]{$\scriptstyle \Psi_\scrZ$}(b);
\draw[->] (A)--node[left]{$\scriptstyle t_i$}(a);
\draw[->] (B)--node[right]{$\scriptstyle \Phi_i$}(b);
\end{tikzpicture}
\]
In this setting all the categories $\scrC_\chamC$, as $\chamC$ varies, are equal.  Further, all functors in the above diagram restrict to the relevant $\scrC$, respectively $\scrC_\chamC$.   Part (1) then follows from \ref{make module}, part (2) by \ref{cor 2 3folds}\eqref{cor 2 3folds 2} since $\scrS_i$ corresponds to $\scrO_{\Curve_i}(-1)$, and the final statement is \ref{cor: no exotic}.
\end{proof}

The following result is slightly more difficult to state, firstly due to the non-uniqueness of partial resolutions, and secondly since unlike the case of $\scrZ$ or $\scrX$, the functor $\Uppsi_{\scrY'}^{-1}\Phi_\upbeta\Uppsi_\scrY$ below does not yet have a good geometric interpretation.   We further remark that in general, the following is the best possible, but afterwards in \ref{ex:D4 B2} and \ref{ex: pure surfaces} we illustrate how in some situations the result can be improved.  
\begin{cor}\label{cor: cpartial groupoid main}
Consider a partial crepant resolution $\scrY\to\Spec\mathbb{C}^2/\Gamma$, and let $x\in\scrC$. 
\begin{enumerate}
\item\label{cor: cpartial groupoid main 1} The following statements are equivalent.
\begin{enumerate}
\item $\Hom_\scrC(x,x[j])=0$ for all $j<0$.
\item There exists $\upbeta$, and a partial resolution $\scrY'\to\mathbb{C}^2/\Gamma$ obtained from $\scrY$ by iterated wall crossing, such that $\Uppsi_{\scrY'}^{-1}\Phi_\upbeta\Uppsi_\scrY(x)$ belongs to perverse sheaves.
\end{enumerate}
\item\label{cor: cpartial groupoid main 2}  Furthermore, the following are equivalent.
\begin{enumerate}
\item $\Hom_\scrC(x,x[j])=0$ for all $j<0$ and $\Hom_\scrC(x,x)\cong\mathbb{C}$.
\item There exists $\upbeta$, a partial resolution $\scrY'\to\mathbb{C}^2/\Gamma$ obtained from $\scrY$ by iterated wall crossing, and some $\Curve_i\subset\scrY'$ such that $\Uppsi_{\scrY'}^{-1}\Phi_\upbeta\Uppsi_\scrY(x)\cong \scrO_{\Curve_i}(-1)$.
\end{enumerate}
\end{enumerate}
\end{cor}
\begin{proof}
(1) is an immediate consequence of \ref{make module}, and (2) follows from \ref{cor 2 3folds}\eqref{cor 2 3folds 2} since $\scrS_i$ corresponds to $\scrO_{\Curve_i}(-1)$. 
\end{proof}

\begin{exa}\label{ex:D4 B2}
Consider the partial resolution $\scrY\to\mathbb{C}^2/\mathrm{D}_4$ given by the Dynkin data $\,\Dfour{P}{P}{B}{B}\,$, where $\scrI$ is the set of black nodes.  That is, $\scrY$ is obtained from the $\mathrm{D}_4$ minimal resolution by contracting the two curves corresponding to the black nodes.  The following illustrates the associated intersection arrangement $\scrH$.
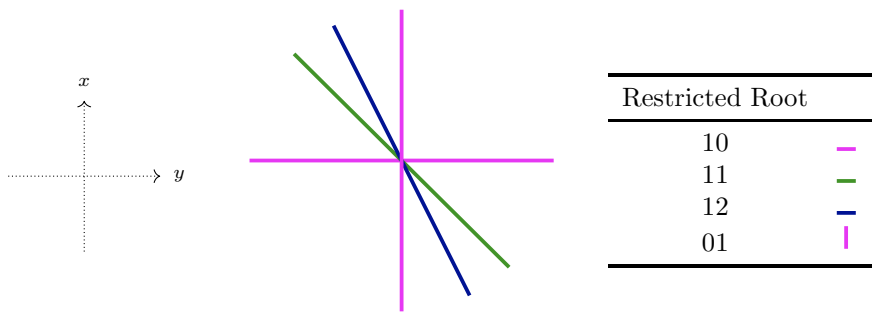
\[
\begin{array}{cccc}
\begin{array}{c}
\begin{tikzpicture}[scale=0.5]
\draw[->,densely dotted] (180:2cm)--(0:2cm);
\node at (0:2.5) {$\scriptstyle y$};
\draw[->,densely dotted] (-90:2cm)--(90:2cm);
\node at (90:2.5) {$\scriptstyle x$};
\end{tikzpicture}
\end{array}
&
\begin{array}{c}
\begin{tikzpicture}[scale=1]
\draw[line width=\mythick mm,Pink] (180:2cm)--(0:2cm);
\draw[line width=\mythick mm,Green] (135:2cm)--(-45:2cm);
\draw[line width=\mythick mm, Blue] (116.57:2cm)--(-63.43:2cm);
\draw[line width=\mythick mm,Pink] (90:2cm)--(-90:2cm);
\end{tikzpicture}
\end{array}&
\begin{array}{c}
\begin{tabular}{ccc}
\toprule
Restricted Root&\\
\midrule
$10$&$\tikz\draw[line width=\mythick mm, Pink] (0,0) -- (0.25,0);$\\
$11$&$\tikz\draw[line width=\mythick mm, Green] (0,0) -- (0.25,0);$\\
$12$&$\tikz\draw[line width=\mythick mm, Blue] (0,0) -- (0.25,0);$\\
$01$&$\tikz\draw[line width=\mythick mm, Pink] (0,-0.15) -- (0,0.15);$\\
\bottomrule
\end{tabular}
\end{array}
\end{array}
\]
In this example, the wall crossing rule of \cite[1.16]{IW9} returns the same indexing set $\,\Dfour{P}{P}{B}{B}\,$ in each chamber and thus in \S\ref{sec: bij MMP} the \emph{identical} category $\scrC_\chamC$ gets assigned to each chamber.  Further, each wall crossing is a twist autoequivalence \cite[10.5]{IW9} over noncommutative deformations of the corresponding curve.  Thus, as in the case of the minimal resolution, the categories can be identified and thus there is a braid group action of $\mathrm{B}_2$ on $\Db(\coh\scrY)$.  The result \ref{cor: cpartial groupoid main}\eqref{cor: cpartial groupoid main 1} then translates (since always $\scrY'=\scrY$) to give the analogue of \ref{thm: minimal main}\eqref{thm: minimal main 1} for the braid group of $\mathrm{B}_2$, and \ref{cor: cpartial groupoid main}\eqref{cor: cpartial groupoid main 2} gives the analogue of \ref{thm: minimal main}\eqref{thm: minimal main 2}.
\end{exa}

\begin{exa}\label{ex: pure surfaces}
Continuing \ref{ex: Dynkin res}, for the Dynkin data $\,\Dfive{B}{P}{B}{P}{B}\,$ the intersection arrangement for $\scrY\to\mathbb{C}^2/\mathrm{D}_5$ is given in \ref{ex: Dynkin res 2}.  Now the wall crossing rule of \cite[1.16]{IW9} returns different indexing sets to the chambers, as follows.
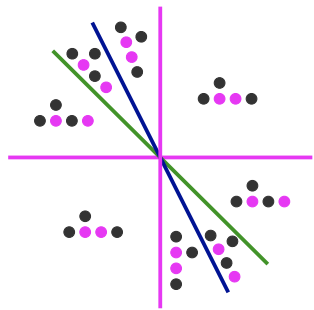
\[
\begin{array}{c}
\begin{tikzpicture}[scale=1]
\draw[line width=\mythick mm,Pink] (180:2cm)--(0:2cm);
\draw[line width=\mythick mm,Green] (135:2cm)--(-45:2cm);
\draw[line width=\mythick mm, Blue] (116.57:2cm)--(-63.43:2cm);
\draw[line width=\mythick mm,Pink] (90:2cm)--(-90:2cm);
\node[rotate=-90] at (-77:1.4cm) {$\Dfive{B}{P}{B}{P}{B}$};
\node[rotate=-60] at (-54:1.55cm) {$\Dfive{B}{P}{B}{B}{P}$};
\node at (-20:1.4cm) {$\Dfive{B}{P}{B}{B}{P}$};
\node at (45:1.25cm) {$\Dfive{B}{P}{B}{P}{B}$};
\node[rotate=-70] at (102:1.5cm) {$\Dfive{B}{P}{B}{P}{B}$};
\node[rotate=-45] at (125:1.5cm) {$\Dfive{B}{P}{B}{B}{P}$};
\node[rotate=0] at (155:1.4cm) {$\Dfive{B}{P}{B}{B}{P}$};
\node at (225:1.25cm) {$\Dfive{B}{P}{B}{P}{B}$};

\end{tikzpicture}
\end{array}
\]
Thus as the chambers vary the categories $\scrC_\chamC$ are sometimes equal, and sometimes are not.  This translates into a mixed braid group action on $\Db(\coh\scrY)$.
\end{exa}

\subsection{\texorpdfstring{$3$}{3}-folds}
In the $3$-fold setting of flopping contractions $\scrX\to\Spec\scrR$, our main interest is in establishing \ref{cor: flop main PBr} below.  However, there seems to be no method of proving that without taking square roots, and first establishing the following result.

\begin{cor}\label{cor: flop main}
Let $\scrX\to\Spec\scrR$ be a $3$-fold flopping contraction, where $\scrX$ has at worst Gorenstein terminal singularities, and consider $x\in\scrC$.
\begin{enumerate}
\item\label{cor: flop main 1} The following statements are equivalent.
\begin{enumerate}
\item $\Hom_\scrC(x,x[j])=0$ for all $j<0$.
\item There exists a sequence of flop functors such that $\Flop_\upbeta(x)$ belongs to perverse sheaves, on a possibly different $\scrX^+\to\Spec\scrR$.
\end{enumerate}
\item\label{cor: flop main 2}  Furthermore, the following are equivalent.
\begin{enumerate}
\item $\Hom_\scrC(x,x[j])=0$ for all $j<0$ and $\Hom_\scrC(x,x)\cong\mathbb{C}$.
\item There exists $\upbeta$ such that $\Flop_\upbeta(x)\cong\scrO_{\Curve_i}(-1)$ for some curve $\Curve_i$ on a possibly different $\scrX^+\to\Spec\scrR$.
\end{enumerate}
\end{enumerate}
Furthermore, the heart of every bounded t-structure on $\scrC$ is the image, under the action of the group generated by the flop functors, of the module category of the contraction algebra of some $\scrX^+\to\Spec\scrR$ obtained from $\scrX\to\Spec\scrR$ by iterated flop.
\end{cor}

\begin{proof}
Under the derived equivalence $\Psi_\scrX$ in \eqref{NCCR Db}, the geometric category $\scrC$ defined in \eqref{eqn: define C intro} corresponds to the $\scrC_\chamC$ in \S\ref{sec: bij MMP}.  Furthermore by \cite[4.2]{HomMMP}, the mutation functors are functorially isomorphic to inverse of the flop functor, via the following commutative diagram.
\[
\begin{tikzpicture}
\node (A) at (0,0) {$\Db(\coh\scrX)$};
\node (B) at (3,0) {$\Db(\fmod\Lambda)$};
\node (a) at (0,-1.5) {$\Db(\coh\scrX^+_i)$};
\node (b) at (3,-1.5) {$\Db(\fmod\upnu_i\Lambda)$};
\draw[->] (A)--node[above]{$\scriptstyle \Psi_\scrX$}(B);
\draw[->] (a)--node[above]{$\scriptstyle \Psi_{\scrX^+_i}$}(b);
\draw[->] (A)--node[left]{$\scriptstyle \Flop^{-1}_i$}(a);
\draw[->] (B)--node[right]{$\scriptstyle \Phi_i$}(b);
\end{tikzpicture}
\]
Part (1) then follows from \ref{make module}, part (2) by \ref{cor 2 3folds}\eqref{cor 2 3folds 2} since $\scrS_i$ corresponds to $\scrO_{\Curve_i}(-1)$, and the final statement is \ref{cor: no exotic}.
\end{proof}

In terms of applications, fixing a category $\scrC$ it is desirable to know for example the spherical objects under the action of the autoequivalence group. The following generalises \cite[6.12(2)]{SW} to all $3$-fold flops.

\begin{cor}\label{cor: flop main PBr}
Let $\scrX\to\Spec\scrR$ be a $3$-fold flopping contraction, where $\scrX$ is at worst Gorenstein terminal, with contraction algebra $\Lambda_{\con}$. For $x\in\scrC$, the following are equivalent.
\begin{enumerate}
\item\label{cor: flop main PBr 1} $\Hom_\scrC(x,x[i])=0$ for all $i<0$ and $\Hom_\scrC(x,x)\cong\mathbb{C}$ 
\item\label{cor: flop main PBr 2} Under the action of $\PBr$ on $\scrC$, $x$ is in the orbit of the (finitely many) brick $\Lambda_{\con}$-modules, or their shifts by $[1]$.
\end{enumerate}
All such $x$ noncommutatively deform to give a spherical twist autoequivalence.
\end{cor}
\begin{proof}
(1)$\Rightarrow$(2) Consider the $\upbeta\colon \chamD\to\chamC$ from \ref{cor: flop main}\eqref{cor: flop main 2} and also the atom $\upalpha\colon\chamD\to\chamC$.  By torsion pairs (see e.g.\ \cite[\S5]{HW}), each $\Phi_\upalpha(\scrS_i)$ is either a brick module, or the shift $[1]$ applied to a brick module.  Hence applying $\Phi_\upbeta\circ \Phi_\upalpha^{-1}$ to either this brick or its shift shows that $x$ is in the desired orbit.  The reverse direction (2)$\Rightarrow$(1) is trivial.  The final statement follows from \ref{cor: flop main}\eqref{cor: flop main 2} since the sheaves $\scrO_{\Curve_i}(-1)$ satisfy the conclusion \cite{DW1,DW3}.
\end{proof}

\begin{rem}
As demonstrated in the example \cite[6.12(2)]{SW},  it is possible to construct a subset of bricks, for which \ref{cor: flop main PBr}\eqref{cor: flop main PBr 2} can be improved to the statement that $x$ is in the orbit under $\PBr$ of this subset, and their shifts by $[1]$.  In general, it is not clear how to construct a subset of minimal size for which \ref{cor: flop main PBr}\eqref{cor: flop main PBr 2} remains true.
\end{rem}

\subsection{Topological Corollaries}
For any triangulated category $\scrT$ there is an associated complex manifold $\Stab\scrT$, whose points are Bridgeland stability conditions on $\scrT$.
Every exact equivalence $\Phi \colon \scrT \to \scrT'$ induces a homeomorphism  $\Phi_* \colon \Stab\scrT \xrightarrow{\sim} \Stab\scrT'$.

Again let $\scrC = \scrC_\chamC$ be the category defined in \S\ref{sec: bij MMP}, and $\mathrm{U}_{\chamC} \subset \Stab \scrC_\chamC$ be the subset of stability conditions with the standard heart.
Since $\mathrm{U}_{\chamC}\cong\mathbb{H}^{\oplus n}$ where $n$ is the number of simples, which is connected, there exists a connected component $\cStab{}\scrC \subset \Stab\scrC$ containing $\mathrm{U}_{\chamC}$.
For a path $\upbeta \colon \chamD \to \chamC$, 
write $\mathrm{U}_{\upbeta} \subset \Stab \scrC_\chamC$ for the image of $\mathrm{U}_\chamD$ under $(\Phi_{\upbeta})_*$, and further write $\mathsf{Term}(\chamC)$ for all morphisms in the Deligne groupoid ending at $\chamC$.

\begin{cor}\label{cor: stability connected}
The set $\{ \mathrm{U}_{\upbeta} \mid \upbeta \in \mathsf{Term}(\chamC)\}$ covers  $\Stab\scrC$. In particular,  $\cStab{}\scrC=\Stab\scrC$.
\end{cor}
\begin{proof}

For a simple mutation $s_i \colon \chamC \to \chamC_i$, the $\mathrm{U}_{\chamC}$ and $\mathrm{U}_{\chamC_i}$ share a codimension one boundary (see e.g.\ \cite[6.3]{HW}, \cite[5.5]{Bstab}).  In particular, the associated homeomorphism $(\Phi_i)_*$ restricts to 
\[ (\Phi_i)_* \colon \cStab{}\scrC_{\chamC} \to \cStab{}\scrC_{\chamC_i}. \]
Let $\upsigma = (Z, \scrA)$ be an arbitrary point in $\Stab\scrC_{\chamC}$.
Then by \ref{cor: no exotic}, there exists a path $\upbeta \colon \chamC \to \chamD$ such that $\Phi_{\upbeta}(\scrA)$ is the standard heart of $\scrC_{\chamD}$.
Now $\Phi_{\upbeta}$ is a composition of simple mutations $\Phi_i$ or their inverses, and thus this also preserves the components $\cStab{}\scrC_{\chamC}$ and $\cStab{}\scrC_{\chamD}$.
Since $(\Phi_{\upbeta})_*(\upsigma) \in \mathrm{U}_\chamD \subset \cStab{}\scrC_{\chamD}$, applying $\Phi_{\upbeta^{-1}}=\Phi_\upbeta^{-1}$ gives $\upsigma \in \mathrm{U}_{\upbeta^{-1}} \cap \,\cStab{}\scrC_{\chamC}$, where $\upbeta^{-1}\colon\chamD\to\chamC$.  Since $\mathrm{U}_{\upbeta^{-1}}$ is connected and $\cStab{}\scrC_{\chamC}$ is a connected component, necessarily $\upsigma \in \mathrm{U}_{\upbeta^{-1}} \subset \cStab{}\scrC_{\chamC}$, which proves the result.
\end{proof}

Consider the group $\Auteq^{\mathrm{FM}}\scrC$ consisting of those $\Upphi|_\scrC$ where $\Upphi$ is a  Fourier--Mukai equivalence $\Db(\coh X)\to\Db(\coh X)$ that commutes with $\mathbf{R}f_*$.  Automatically  $\Upphi|_\scrC\colon\scrC\to\scrC$.

\begin{cor}\label{cor: FM char}
Suppose that $\scrX\to\Spec \scrR$ is a $3$-fold flop, where $\scrX$ has at worst terminal singularities.  Then $\Auteq^{\mathrm{FM}}\scrC\cong\PBr\scrC$.
\end{cor}
\begin{proof}
Consider the subgroup $\cAut{}\scrC$ of $\Auteq^{\mathrm{FM}}\scrC$ that preserves $\cStab{}\scrC$.  By  \cite[7.1]{HW2} it is known that $\cAut{}\scrC\cong\PBr\scrC$. But now using \ref{cor: stability connected} $\cStab{}\scrC=\Stab\scrC$, hence every autoequivalence preserves $\cStab{}\scrC$ and thus $\cAut{}\scrC=\Auteq^{\mathrm{FM}}\scrC$.
\end{proof}

\section{Silting Discrete Derived Categories}\label{sec: silting section}

This section generalises the argument in \S\ref{sec: classification} to the setting of derived categories $\Db(\fmod A)$ where $A$ is a silting discrete finite dimensional algebra.  This gives applications to representation theory.

\subsection{Silting Discrete Preliminaries}\label{sec: silting prelims}
Throughout this section, $A$ will be a finite dimensional $k$-algebra, which to be consistent with the earlier parts of this paper requires $k=\mathbb{C}$.  As is standard, this can be easily be generalised to other fields by weakening the condition on the endomorphism ring of a brick to be a division ring, rather than the base field.

\medskip
We briefly recall silting objects and simple minded collections, following \cite{AI, KY}.  Set $\scrK=\Kb(\proj A)$ and $\scrT=\Db(\fmod A)$.  An object $x\in \scrK$ is called silting if $\Hom_{\scrK}(x,x[i])=0$ for all $i>0$, and further $\mathrm{thick}(x)=\scrK$.  By contrast, a collection $\{y_1,\hdots,y_n\}$ of objects of $\scrT$ are called a simple minded collection (smc) if $\Hom_\scrT(y_i,y_j[t])=0$ for all $i,j$ and all $t<0$, $\Hom_\scrT(y_i,y_j)\cong\mathbb{C}\updelta_{ij}$,  and $\mathrm{thick}(y_1,\hdots,y_n)=\scrT$.  We will often blur the distinction between the set $\{y_1,\hdots,y_n\}$ and the single object $y=\bigoplus y_i$.

The standard example of a silting object is $A$, considered as a complex in degree zero, and the standard example of an smc is $\scrS=\bigoplus \scrS_i$, the sum of the simple $A$-modules.

Below, we will often use the Koenig--Yang bijections. For the categories $\scrK$ and $\scrT$ above, there are bijections \cite[\S5]{KY}
\begin{align}
\{ \mbox{basic silting objects in }\scrK\}&\longleftrightarrow \{ \mbox{bounded t-structures in }\scrT\mbox{ with length heart}\}\label{KYbij}\\
&\longleftrightarrow \{ \mbox{smcs in }\scrT\}\nonumber
\end{align}
where an object is \emph{basic} if there are no repetitions in its Krull--Schmidt decomposition.  The following is the standard definition, translated into the notation of \ref{not: 5.1}.

\begin{dfn}
Let $\scrH$ be the heart of a bounded t-structure on a triangulated category $\scrT$.  An smc $\scrU$ is called 2-term with respect to $\scrH$ if $\scrU\in [-1,0]_\scrH$.  In $\Db(A)$, the set of 2-term smcs with respect to the standard heart of $\Db(A)$ will be written $2\dsmc{A}$.  Under \eqref{KYbij}, the corresponding set of silting objects will be written $2\silt A$.
\end{dfn}

\begin{nota}\label{nota: 5.4}
If $\scrV$ is an smc, we will abuse notation in \textnormal{\ref{not: 5.1}} and write $[a,b]_\scrV=[a,b]_{\scrH_\scrV}$ where $\scrH_\scrV$ is the heart of the bounded t-structure corresponding to $\scrV$ under the Koenig--Yang bijection \eqref{KYbij}. Note that with this abuse, $\scrU\in 2\dsmc A$ if and only if $\scrU\in [-1,0]_{\scrS}$, which by \ref{bound for hearts} is equivalent to $\scrS\in[0,1]_\scrU$.  This is the formulation used in \ref{thm: silting main} below.
\end{nota}

As in \ref{notation via vanishings} and \ref{bound for objects}, the notation \ref{nota: 5.4} can be interpreted in terms of the $\Hom$ vanishings.

\begin{lem} \label{notation via vanishings2}
Let $\scrU$ be an smc, and $x\in\scrT = \Db(A)$.
Then the following are equivalent.
\begin{enumerate}
\item $x \in [a,b]_{\scrU}$.
\item $\Hom(\scrU, x[i]) = 0$ for all $i < a$ and $\Hom(x, \scrU[i]) = 0$ for all $i < -b$.
\end{enumerate}
\end{lem}

\begin{dfn}\label{def: 5.5}\cite[2.4]{AM}
$A$ is called \emph{$\uptau$-tilting finite} if the set $2\dsmc{A}$, equivalently  $2\silt A$, is finite. If for all smcs $\scrU$ there are only finitely many smcs $\scrV$ for which $\scrV\in [0,1]_{\scrU}$, then $A$ is called \emph{silting discrete}.
\end{dfn}
It is clear that any silting discrete algebra is $\uptau$-tilting finite.

\begin{dfn}\label{dfn: semibrick complex}
A basic $x\in\Db(A)$ is a \emph{semibrick complex} if $\Hom_{\Db(A)}(x,x[j])=0$ for all $j<0$ and further $\Hom_{\Db(A)}(x_i,x_j)\cong\mathbb{C}\updelta_{ij}$ for all indecomposable summands $x_i,x_j$ of $x$.  A \emph{brick complex} is a semibrick complex which is indecomposable.  If $x$ happens to be an $A$-module in homological degree zero, we will emphasise this by calling $x$ a semibrick module, or brick module, respectively.  
\end{dfn}

The following is known \cite{KY, David, IJY}. 
\begin{lem} \label{realisation}
Let $A$ be a silting discrete finite dimensional algebra, and $\scrH$ be the heart of a bounded t-structure of $\Db(A)$. Then the following statements hold.
\begin{enumerate}
\item\label{realisation 1} There exists a finite dimensional algebra $\Gamma$ such that $\scrH\simeq \fmod \Gamma$.
\item\label{realisation 2} The realisation functor $\mathsf{real} \colon\Db(\Gamma) \to \Db(A)$ induces an injective map from the set $2\dsmc{\Gamma}$ to the set of smcs in $\Db(A)$ which are $2$-term with respect to $\scrH$.
\item\label{realisation 3} $\Gamma$ is a $\uptau$-tilting finite algebra.
\end{enumerate}
\end{lem}
\begin{proof}
(1) Since $A$ is silting discrete, any heart of a bounded t-structure has finite length \cite[Thm A]{David}.  Under the Koenig--Yang bijection \cite[6.2]{KY} it follows that $\scrH=\fmod\Gamma$ for some $\Gamma=\End_{\Kb(\proj A)}(M)$, where $M\in\Kb(\proj A)$ is silting.\\
(2) As explained in \cite[7.8]{KY}, the induced map
\[ 
\Hom_{\Gamma}(x, y[i]) \to \Hom_{\Db(A)}(\mathsf{real}(x), \mathsf{real}(y)[i]) 
\]
is bijective for all $x, y \in \fmod \Gamma$ and $i \leq 1$.
Let $\scrU \in 2\dsmc{\Gamma}$, say $\scrU= \{y_1, \dots, y_r\}$.
Then by \cite[4.11]{BY}, each $y_i$ is either in $\fmod \Gamma$ or $(\fmod \Gamma)[1]$. Therefore, setting $z_i\colonequals \mathsf{real}(y_i)$, the direct sum $z = \bigoplus_{i=1}^r z_i$ remains a semibrick complex. Since each $z_i$ is either in $\scrH$ or $\scrH[1]$, clearly $z$ is $2$-term with respect to $\scrH$.  

Since $\scrU$ is an smc, by definition $\mathrm{thick}(\scrU)=\Db(\Gamma)$ and thus $\fmod\Gamma\subset \mathrm{thick}(\scrU)$. Applying the triangulated functor $\mathsf{real}$ then gives
\[
\scrH=\mathsf{real}(\fmod\Gamma)\subset \mathsf{real}(\mathrm{thick}(\scrU))\subseteq \mathrm{thick}(\mathsf{real}(\scrU))\colonequals\mathrm{thick}(z)
\]
Thus $\mathrm{thick}(z)$ contains the heart of a bounded t-structure on $\Db(A)$, and so necessarily $\mathrm{thick}(z)=\Db(A)$.  This shows that $z$ is also an smc, and so $\mathsf{real}$ has induced the claimed map out of $2\dsmc{\Gamma}$.  Since $\mathsf{real}$ is an equivalence on $\fmod\Gamma\to \scrH$, and thus on $(\fmod\Gamma)[1]\to \scrH[1]$, it follows that the induced map is injective.\\
(3) Consider the silting object $M$ in (1), so that $\Gamma=\End_{\Kb(\proj A)}(M)$.  By \cite[0.2]{IJY} there is a bijection between silting objects in $M*M[1]$ and the set of support $\uptau$-tilting modules $\stautilt\Gamma$ (see e.g.\ \cite{AIR}).  But as is standard \cite[\S1.2]{David}, a silting object $T\in M*M[1]$ if and only if $M\geq T\geq M[1]$, and so since $A$ is silting discrete, this set is finite.  Hence the set $\stautilt\Gamma$ is finite.
\end{proof}

The following consequence of \ref{realisation} is also somewhat implicit in the literature.  As notation, write $\sbrick A$ for the set of all (basic) semibrick modules.

\begin{cor}\label{cor: complete if in heart}
Let $A$ be silting discrete, and suppose that $x$ is a semibrick complex in $\Db(A)$ that is contained in the heart $\scrH$ of a bounded t-structure.
Then there exists a simple minded collection $\scrU$ in $\Db(A)$ that contains all indecomposable summands of $x$ and is $2$-term with respect to $\scrH$.
\end{cor}
\begin{proof}
By \ref{realisation}, $\scrH\simeq\fmod\Gamma$ where $\Gamma$ is $\uptau$-tilting finite. Let $y \in \fmod \Gamma$ be the semibrick corresponding to $x\in\scrH$ under the equivalence. Since $\Gamma$ is $\uptau$-tilting finite, every torsion class in $\fmod\Gamma$ is functorially finite \cite[1.2]{DIJ}.  Thus the `functorially finite' hypothesis can be dropped from the statement in Asai \cite[2.3]{Asai}, so that \cite[2.3]{Asai} then asserts there is a bijection
\[
2\dsmc\Gamma\to \sbrick \Gamma
\]
given by $\scrX\to\scrX\cap\fmod\Gamma$.  Thus there exists $\scrU\in 2\dsmc{\Gamma}$ such that $\scrU$ contains all the indecomposable summands of $y$.  It follows from \ref{realisation}\eqref{realisation 2} that $\mathsf{real}(\scrU)$ gives the desired smc.
\end{proof}

The challenge, and the content of the next subsection, is to prove that any semibrick complex in $\Db(A)$ is automatically contained in the heart $\scrH$ of a bounded t-structure.  In fact, this works much more generally.

\subsection{Main Result}
Throughout this section, as above, $A$ is a finite dimensional $\mathbb{C}$-algebra.

\begin{dfn}
If $\scrU, \scrV$ are smcs in $\Db(A)$, write $\scrU\geq \scrV$ if $\Hom(\scrV, \scrU[j]) = 0$ for all $j<0$. 
\end{dfn}

Note that, by \ref{notation via vanishings2}, $\scrU \geq \scrV$ if and only if $\scrV \in [a, 0]_{\scrU}$ for some $a \leq 0$. The relation~$\geq$ endows the set $\textsf{smc}\Db(A)$ with the structure of a partially ordered set \cite[7.9]{KY}.

To prove the main result requires the ability to mutate simple minded collections.

\begin{dfn}[\cite{KY}] 
Let $\scrU = \{ y_1, \dots, y_r\}$ be an smc in the derived category $\Db(A)$ of a finite dimensional algebra $A$, and choose $y_i$ in $\scrU$.
The \textit{left mutation} $\upnu_i\scrU \colonequals \{y'_1, \dots, y'_r\}$ of $\scrS$ at $y_i$ is defined as follows.
\begin{enumerate}
\item If $j \neq i$, then $y'_j$ is defined to be the cone of the minimal left approximation $y_j[-1] \to y_{ij}$ in the extension closure of $y_i$.
\item If $j=i$, then $y'_i \colonequals y_i[1]$
\end{enumerate}
\end{dfn}
By \cite[7.8]{KY}, the new collection $\upnu_i\scrU$ is again an smc.
Note that $\scrU > \upnu_i\scrU$ holds since $\upnu_i\scrU \in [-1,0]_{\scrU}$ by construction.
Right mutations are defined similarly, but below we will only require left mutations.

\begin{lem}\label{lem 6.10}
Set $\scrT=\Db(A)$, let $\scrU=\{ y_1, \dots, y_r\} \subset \scrT$ be an smc, and suppose that $x \in \scrT$ satisfies $x \in [a, b]_{\scrU}$. Then the following statements hold. 
\begin{enumerate}
\item $x \in [a, b+1]_{\upnu_i\scrU}$ for all $i$.
\item If $\Hom_\scrT(\H_{\scrU}^b(x), y_i) = 0$, then $x \in [a, b]_{\upnu_i\scrU}$.
\end{enumerate}
\end{lem}
\begin{proof}
(1) Fix any $i$ and put $\upnu_i\scrU = \{y'_1, \dots, y'_r\}$.
By \ref{notation via vanishings2}, proving $x \in [a, b+1]_{\upnu_i\scrU}$ is equivalent to showing that
\begin{enumerate}
\item[(a)] $\Hom_\scrT(\bigoplus_{j=1}^r y'_j, x[a+k]) = 0$ for all $k < 0$, and
\item[(b)]  $\Hom_\scrT(x, \bigoplus_{j=1}^r y'_j[-b -1 + k]) = 0$ for all $k < 0$.
\end{enumerate}
Since $y'_j$ is contained in the extension closure of $y = \bigoplus_{l=1}^r y_l$ for all $j \neq i$,
and since $y'_i = y_i[1]$,
the first vanishing (a) follows from the assumption $x \in [a, b]_{\scrU}$.
Again since $y'_j$ is contained in the extension closure of $y$ for all $j \neq i$,
the vanishing $\Hom_\scrT(x, \bigoplus_{j \neq i} y_j[-b + k]) = 0$ holds for any $k < 0$.
For the remaining case, there is an equality
\[ 
\Hom_\scrT(x, y'_i[-b-1+k]) = \Hom_\scrT(x, y_i[-b+k]) 
\]
by definition, and the RHS is zero for all $k < 0$ by the assumption $x \in [a,b]_{\scrU}$.\\
\noindent
(2) The proof of (1) also shows that $x \in [a, b]_{\upnu_i\scrU}$ if and only if $\Hom(x, y_i[-b]) = 0$.
Since $x \in [a,b]_{\scrU}$ and $y$ is an object in the heart corresponding to $\scrU$, applying \ref{magic lemma Db 2} yields an isomorphism $\Hom(x, y_i[-b]) \cong \Hom(\H_{\scrU}^b(x), y_i)$, which proves the result.
\end{proof}

The following is the main result of this section.

\begin{thm}\label{thm: silting main}
If $A$ is silting discrete, $\scrT=\Db(A)$ and $x\in\scrT$, then the following statements are equivalent.
\begin{enumerate}
\item $\Hom_\scrT(x,x[i])=0$ for all $i<0$.
\item $x$ belongs to the heart of a bounded t-structure.
\end{enumerate}
\end{thm}
\begin{proof}
The only non-trivial direction is that (1) implies (2).
Choose an smc $\scrU$, then there exists integers $a\leq b$ such that $x \in [\![a,b]\!]_{\scrU}$.
If $a=b$, then there is nothing to prove.
Thus we assume $a<b$, and claim that there is another smc $\scrV$ such that $x \in [a+1, b]_{\scrV}$.
If this is true, then by induction we can find an smc whose corresponding heart contains the object~$x$.

To prove the claim, consider
\[
\Updelta= \{ \scrW \in \smc\Db(A) \mid \scrU \in [0,1]_{\scrW} ~\text{and}~ x \in [a,b]_{\scrW} \}.
\]
This is a poset with respect to $\geq$.  Further, the poset $\Updelta$ is finite since $A$ is silting discrete (see \ref{def: 5.5}), 
and thus there exists a minimal element  $\scrV = \{y_1, \dots, y_r \}$.
We will show that this $\scrV$ satisfies the desired property.
Assume, for the aid of a contradiction, that $\H^a_\scrV(x) \neq 0$, and choose a simple $y_i \hookrightarrow \H_{\scrV}^a(x)$.
By \ref{magic lemma Db 2}, since $x \in [\![a,b]\!]_{\scrV}$ and $\scrU \in [0,1]_{\scrV}$, necessarily
\[ 
\Hom(\scrU, x[a-1]) \cong\Hom(\H_{\scrV}^1(\scrU), \H_{\scrV}^a(x)).
\]
The LHS is zero by \ref{notation via vanishings2} since $x \in [\![a,b]\!]_{\scrU}$.
Similarly, since $a<b$ and $x$ has no negative Exts, it also follows from \ref{magic lemma Db 2} that $\Hom(\H_{\scrV}^b(x), \H_{\scrV}^a(x)) = 0$.
Then our choice of $y_i$ together with these vanishings yields
\[ 
\Hom(\H_{\scrV}^1(\scrU), y_i) = 0 ~ \text{and} ~ \Hom(\H_{\scrV}^b(x), y_i) = 0. 
\]
By \ref{lem 6.10}, these imply that $\scrU \in [0,1]_{\upnu_i\scrV}$ and $x \in [a,b]_{\upnu_i\scrV}$, respectively.
But then $\upnu_i\scrV$ is an smc in $\Updelta$ with $\scrV > \upnu_i\scrV$, which contradicts the minimality of $\scrV$.
\end{proof}

Recall that a semibrick complex $x$ is said to have maximal rank if the number of indecomposable summands of $x$ equals the number of simple $A$-modules.
The following is immediate from \ref{thm: silting main}, where part (3) generalises \cite{DIJ}.

\begin{cor}\label{cor: silting simples}
If $A$ is silting discrete, then the following statements hold.
\begin{enumerate}
\item\label{cor: silting simples 1} Any semibrick complex is a subcollection of a simple minded collection.
\item\label{cor: silting simples 2} Any semibrick complex of maximal rank is a simple minded collection.
\end{enumerate}
\end{cor}

\begin{proof}
(1) Let $x$ be a semibrick complex.
Then \ref{thm: silting main} implies that there exists the bounded heart $\scrH$ that contains $x$, so the result is a special case of \ref{cor: complete if in heart}.\\
\noindent
(2) Let $x$ be a semibrick complex of maximal rank.
By (1), $x$ is a subcollection of a simple minded collection.  But by \cite[5.5]{KY} every simple minded collection has maximal rank, thus  $x$ itself is a simple minded collection.
\end{proof}

For the next corollary, recall that $\sbrick A$ is the set of all semibrick modules. If $x,y\in\sbrick A$, then $x\oplus y[1]$ is called a semibrick pair if $\Hom_A(x,y)=0=\Ext^1_A(x,y)$.
Note that automatically a semibrick pair $x \oplus y[1]$ is a semibrick complex.

Part (3) of the following answers a general question of \cite{BH, HI}.
\begin{cor}\label{cor: semi of maximal rank}
If $A$ is silting discrete, then the following statements hold.
\begin{enumerate}
\item\label{cor: semi of maximal rank 1} Every semibrick pair is a semibrick in the heart $\scrH$ of a bounded t-structure satisfying $\scrH\in [-1,0]$ with respect to the standard t-structure.
\item\label{cor: semi of maximal rank 2} Every semibrick pair is a subset of the simples in the heart $\scrA$ of a bounded t-structure satisfying $\scrA\in [-2,0]$ with respect to the standard t-structure.
\item\label{cor: semi of maximal rank 3}   Every semibrick pair of maximal rank is a $2$-term simple minded collection.
\end{enumerate}
\end{cor}

\begin{proof}
Let $x \oplus y[1]$ be a semibrick pair.\\
\noindent
(1) Let  $\scrS$ be the smc consisting of simple $A$-modules, so that the heart corresponding to $\scrS$ is the standard heart.
Consider the finite partially ordered set
\[ \Updelta= \{ \scrW \in \smc\Db(A) \mid \scrS \in [0,1]_{\scrW} ~\text{and}~ x \oplus y[1] \in [-1,0]_{\scrW} \}. \]
Then the proof of \ref{thm: silting main} shows that a minimal element $\scrU$ of $\Updelta$ satisfies $x \in [0,0]_{\scrU} = \scrH_{\scrU}$,
where $\scrH_{\scrU}$ is the heart of the corresponding bounded t-structure.
Since $\scrS \in [0,1]_{\scrH_{\scrU}}$ by definition, \ref{nota: 5.4} shows that $\scrH_\scrU\in [-1,0]$ with respect to the standard t-structure.\\
\noindent
(2) Let $\scrH$ be the heart of the t-structure in (1).
Then \ref{cor: complete if in heart} allows us to find the heart $\scrA$ of another t-structure such that $\scrA \in [-1,0]_{\scrH}$,
and the smc corresponding to $\scrA$ contains all indecomposable summands of $x \oplus y[1]$.
Each $a\in[-1,0]_{\scrH}$ and so has a 2-term filtration as in \eqref{triangles for y}, thus applying  $\H^*$ with respect to the standard t-structure and using $\scrH \in [-1,0]$, it is easy to see that $\scrA\in [-2,0]$.\\
\noindent
(3) This follows from (1) and \ref{cor: silting simples}.
\end{proof}

\subsection{Special Case: Contraction Algebras}\label{subsec contraction cor}
A large class of silting discrete algebras are the contraction algebras $\Lambda_{\con}$ from \S\ref{subsec: cont alg}.  The associated category $\Db(\Lambda_{\con})$ is of particular interest, since it conjecturally is the classifying object of smooth flops.  As such, $\Db(\Lambda_{\con})$ should not exhibit any exotic behaviour, nor contain any `unexpected' objects.  The following result confirms this, adding to the existing evidence in \cite{August}.

In what follows, $F_\upbeta$ are compositions of the standard equivalences in \eqref{Jenny comm diag}, which fixing the same start and end point, generate a subgroup $\PBr$ of $\Auteq\Db(\Lambda_{\con})$.
\begin{cor}\label{cor: semibrick for cont alg}
Let $\Lambda_{\con}$ be the contraction algebra associated to a $3$-fold flopping contraction $\scrX\to\Spec\scrR$, where $\scrX$ has at worst Gorenstein terminal singularities, and let $x\in\Db(\Lambda_{\con})$.  Then the following statements hold.
\begin{enumerate}
\item If $x$ is a semibrick complex, then there exists $\upbeta$, and a sum of simple modules $\bigoplus_{i\in\scrI}\scrS_i$ on a possibly different contraction algebra $\Upgamma_{\con}$, such that $F_\upbeta(\bigoplus_{i\in\scrI}\scrS_i)\cong x$.
\item Every semibrick complex in $\Db(\Lambda_{\con})$ is in the orbit, under the action of $\PBr$, of the (finitely many) semibrick modules or their shift by $[1]$.
\end{enumerate}
\end{cor}
\begin{proof}
(1) By \ref{cor: silting simples}\eqref{cor: silting simples 1} the object $x$ is a subcollection of a simple minded collection and thus by the KY bijections $x$ is a sum of simples in the heart of a bounded t-structure. But by \cite[3.16]{August1} all such t-structures are of the form $F_\upbeta(\fmod\Gamma_{\con})$ for some $\upbeta$.\\
(2) This follows from (1), in the same way as in \ref{cor: flop main PBr}. Say $\upbeta\colon \chamD\to\chamC$, then consider the atom $\upalpha\colon\chamD\to\chamC$.  By torsion pairs, each $F_\upalpha(\scrS_i)$ is either a brick module, or the shift $[1]$ applied to a brick module.  Hence $F_\upbeta\circ F_\upalpha^{-1}$ shows that $x$ is in the desired orbit.
\end{proof}

\begin{rem}
In the flops setting above, the previous results \ref{cor: flop main} and \ref{cor: flop main PBr} classify the semibrick complexes in the category $\scrC$, whilst \ref{cor: semibrick for cont alg} classifies the semibrick complexes in the category $\Db(\Lambda_{\con})$.   The fact that these results are so similar is all the more remarkable, given the categories $\scrC$ and $\Db(\Lambda_{\con})$ are never equivalent. 

Indeed, it follows from the description of all bounded t-structures in \cite{August} that the category $\Db(\Lambda_{\con})$ enjoys the following property: the heart $\scrH$ of every bounded t-structure on $\Db(\Lambda_{\con})$ satisfies $\Ext^t_{\scrH}(a,b)\cong\Hom_{\Db(\Lambda_{\con})}(a,b[t])$ for all $a,b\in\scrH$ and all $t\in\mathbb{Z}$.  We claim that $\scrC$ does not satisfy this property, and so  $\scrC$ and $\Db(\Lambda_{\con})$ are not equivalent. 

As before, consider the bounded heart $\scrH\colonequals \fmod\Lambda_{\con}$ on $\scrC$.  Writing $\scrS$ for the sum of the simple $\Lambda_{\con}$-modules, then since $\Lambda_{\con}$ is projective,
\[
\Ext^3_{\scrH}(\Lambda_{\con},\scrS)
=
\Ext^3_{\fmod\Lambda_{\con}}(\Lambda_{\con},\scrS)=0.
\]
But as in the setup (\S\ref{sec: bij MMP}), $\scrC\subseteq\Db(\Lambda)$, thus 
\[
\Hom_{\scrC}(\Lambda_{\con},\scrS[3])=
\Hom_{\Db(\Lambda)}(\Lambda_{\con},\scrS[3])=\Ext^3_\Lambda(\Lambda_{\con},\scrS),
\] 
which is non-zero by \cite[4.7]{DW3}, as a consequence of the 3-sCY property.
\end{rem}

\subsection{Further Remarks}
The question of whether a semibrick complex can be completed to a 2-simple minded collection has been addressed elsewhere (e.g.\ \cite{BH, HI}), but the appearance of the interval $[-2,0]$ in \ref{cor: semi of maximal rank} gives some evidence that this is the wrong question.

\begin{exa}\label{ex: no contra}
Consider the projective algebra $\Uppi$ of type $D_4$, which is silting discrete by \cite{AD}, and the explicit semibrick pair $(M\oplus N)\oplus E[1]$ constructed in \cite[4.0.8]{BH}.  As shown in \emph{loc.\ cit.}, this does not complete to a 2-smc.  However by \ref{cor: semi of maximal rank}\eqref{cor: semi of maximal rank 2} it does complete to an smc in the region $[-2,0]$.  This shows that the region $[-2,0]$ cannot be improved to $[-1,0]$.
\end{exa}

In fact, the proof of \ref{cor: semi of maximal rank}\eqref{cor: semi of maximal rank 2} shows that the full strength of silting discrete is not required to answer the  question of whether a semibrick pair completes to an smc.  The following example illustrates this in the case of the preprojective algebra $\Uppi$ of type $A$, where it is still not known whether $\Uppi$ is silting discrete.

\begin{exa}\label{ex: no contra 2}
Consider the preprojective algebra $\Uppi$ of type $A_3$, namely the path algebra of the quiver
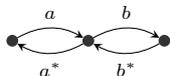
\[
\begin{tikzpicture}[>=stealth]
\node (A) at (0,0) [B] {};
\node (B) at (1,0) [B] {};
\node (C) at (2,0) [B] {};
\draw[bend left,->] (A) to node[above]{$\scriptstyle a$}(B);
\draw[bend left,->] (B) to node[below]{$\scriptstyle a^*$}(A);
\draw[bend left,->] (B) to node[above]{$\scriptstyle b$}(C);
\draw[bend left,->] (C) to node[below]{$\scriptstyle b^*$}(B);
\end{tikzpicture}
\]
subject to the relations $aa^*=0$, $a^*a=bb^*$ and $b^*b=0$.  Set $\scrT=\Db(\fmod\Uppi)$ and
\[
x\colonequals
\bigl(
\kern -4pt
\begin{array}{c}
\begin{tikzpicture}[>=stealth]
\node (A) at (0,0)  {$\mathbb{C}$};
\node (B) at (1,0) {$\mathbb{C}$};
\node (C) at (2,0)  {0};
\draw[bend left,->] (A) to node[above]{$\scriptstyle 1$}(B);
\draw[bend left,->] (B) to node[below]{$\scriptstyle 0$}(A);
\draw[bend left,->] (B) to node[above]{$\scriptstyle $}(C);
\draw[bend left,->] (C) to node[below]{$\scriptstyle $}(B);
\end{tikzpicture}
\end{array}
\kern -4pt
\bigr)
\oplus 
\bigl(
\kern -4pt
\begin{array}{c}
\begin{tikzpicture}[>=stealth]
\node (A) at (0,0)  {$\mathbb{C}$};
\node (B) at (1,0) {$\mathbb{C}$};
\node (C) at (2,0)  {$\mathbb{C}$};
\draw[bend left,->] (A) to node[above]{$\scriptstyle 1$}(B);
\draw[bend left,->] (B) to node[below]{$\scriptstyle 0$}(A);
\draw[bend left,->] (B) to node[above]{$\scriptstyle 1$}(C);
\draw[bend left,->] (C) to node[below]{$\scriptstyle 0$}(B);
\end{tikzpicture}
\end{array}
\kern -4pt
\bigr)[1].
\]
As can be checked by hand, using the fact that $\Uppi$ has finite representation type, or by using GAP \cite{QPA}, $x$ satisfies $\Hom_\scrT(x,x[i])=0$ for all $i<0$, and is a semibrick pair.  It does not complete to a 2-smc.  But consider the 2-smc
\[
\scrS'\colonequals
\bigl(
\kern -4pt
\begin{array}{c}
\begin{tikzpicture}[>=stealth]
\node (A) at (0,0)  {$\mathbb{C}$};
\node (B) at (1,0) {$\mathbb{C}$};
\node (C) at (2,0)  {0};
\draw[bend left,->] (A) to node[above]{$\scriptstyle 1$}(B);
\draw[bend left,->] (B) to node[below]{$\scriptstyle 0$}(A);
\draw[bend left,->] (B) to node[above]{$\scriptstyle $}(C);
\draw[bend left,->] (C) to node[below]{$\scriptstyle $}(B);
\end{tikzpicture}
\end{array}
\kern -4pt
\bigr)
\oplus 
\bigl(
\kern -4pt
\begin{array}{c}
\begin{tikzpicture}[>=stealth]
\node (A) at (0,0)  {$\mathbb{C}$};
\node (B) at (1,0) {$0$};
\node (C) at (2,0)  {$0$};
\draw[bend left,->] (A) to node[above]{$\scriptstyle $}(B);
\draw[bend left,->] (B) to node[below]{$\scriptstyle $}(A);
\draw[bend left,->] (B) to node[above]{$\scriptstyle $}(C);
\draw[bend left,->] (C) to node[below]{$\scriptstyle $}(B);
\end{tikzpicture}
\end{array}
\kern -4pt
\bigr)[1]
\oplus 
\bigl(
\kern -4pt
\begin{array}{c}
\begin{tikzpicture}[>=stealth]
\node (A) at (0,0)  {$0$};
\node (B) at (1,0) {$\mathbb{C}$};
\node (C) at (2,0)  {$\mathbb{C}$};
\draw[bend left,->] (A) to node[above]{$\scriptstyle $}(B);
\draw[bend left,->] (B) to node[below]{$\scriptstyle $}(A);
\draw[bend left,->] (B) to node[above]{$\scriptstyle 1$}(C);
\draw[bend left,->] (C) to node[below]{$\scriptstyle 0$}(B);
\end{tikzpicture}
\end{array}
\kern -4pt
\bigr)[1]
\]
In particular, $\scrS'$ is the sum of the simples in a certain heart $\scrH$ of a bounded t-structure, which is intermediate between $\fmod\Uppi$ and $\fmod\Uppi[1]$.   Again, either by hand or by GAP \cite{QPA}, it can be seen that
\[
\Hom_\scrT( \scrS', x[i] ) = 0 = \Hom_\scrT(x, \scrS'[i])
\]
for all $i<0$, and so since every object of $\scrH$ is filtered by $\scrS'$, it follows that
\[
\Hom_\scrT( \scrH, x[i] ) = 0 = \Hom_\scrT(x, \scrH[i])
\]
for all $i<0$.  Since $\scrH$ is the heart of a bounded t-structure, $x\in\scrH$.  Thus, $x$ is a semibrick in the heart of a bounded t-structure which is intermediate between $\fmod\Uppi$ and $\fmod\Uppi[1]$.  Whilst $x$ is not a sum of simples in this heart, if this heart is $\uptau$-tilting finite (this seems to be conjecturally true in the literature), then the proof of \ref{cor: complete if in heart} can be used to show that $x$ is the simples in a different heart $\scrH'$, for which $\scrH'\in[-2,0]$.
\end{exa}

\end{document}